\documentclass[12pt, twoside, reqno]{amsart}
\usepackage{latexsym}
\usepackage{amsmath}
\usepackage{amssymb}
\usepackage{amsthm}
\usepackage{amscd}
\usepackage{enumerate} 
\usepackage{amssymb} 
\usepackage{mathrsfs}
\usepackage[all]{xy}
\usepackage{tikz}
\usepackage{color}
\usepackage{graphicx}
\usepackage{mathtools}
\usepackage{comment}
\usepackage{multirow}
\usepackage{comment}
\usepackage
{hyperref}
\usepackage[foot]{amsaddr}
\usepackage{bm}
\hypersetup{colorlinks=true}

\title{A review of Lacini's classification of rank one log del Pezzo surfaces in characteristic different from two and three}

\author{Masaru Nagaoka}
\email{masaru.nagaoka@gakushuin.ac.jp}
\address{Gakushuin University, 1-5-1 Mejiro, Toshima-ku, Tokyo 171-8588, Japan}

\def\phi{\varphi}
\def\epsilon{\varepsilon}
\def\tilde{\widetilde}

\def\Bs{\operatorname{Bs}}

\def\Dyn{\operatorname{Dyn}}
\def\Sing{\operatorname{Sing}}
\def\Char{\operatorname{Char}}
\def\Gap{\operatorname{Gap}}
\def\mult{\operatorname{mult}}

\newcommand{\Q}{\mathbb{Q}} 
\newcommand{\C}{\mathbb{C}} 
 
\newcommand{\Z}{\mathbb{Z}}
\newcommand{\PP}{\mathbb{P}}

\newcommand{\FF}{\mathbb{F}}
\newcommand{\ZZ}{\mathbb{Z}}

\newcommand{\wt}{\widetilde}

\theoremstyle{plain}
\newtheorem{thm}{Theorem}[section] 

\newtheorem{prop}[thm]{Proposition}

\newtheorem{lem}[thm]{Lemma}
\theoremstyle{definition} 
\newtheorem{defn}[thm]{Definition}

\newtheorem{eg}[thm]{Example} 

\theoremstyle{remark}
\newtheorem{rem}[thm]{Remark}

\newtheorem{defn and notation}[thm]{Definition and Notation}
 
\newtheorem{nota}[thm]{Notation}


\begin{document}
\tolerance = 9999

\begin{abstract}
   This paper reviews Lacini's classification \cite{Lac} of log del Pezzo surfaces of rank one in characteristics different from two and three, with a focus on where and how Lacini enhanced the techniques of Keel and M\textsuperscript{c}Kernan \cite{KM}.
   We point out that there is at most one log del Pezzo surface that may have been erroneously omitted from the list in \cite[\S 6.1]{Lac}, namely the surface in Example \ref{eg:cexLac4.19}.
   We also extend the results of \cite[\S 4.2]{Lac} to arbitrary characteristic.
\end{abstract}

\maketitle
\markboth{Masaru Nagaoka}{Notebook}
\tableofcontents

\section{Introduction}
Throughout this paper, we work over an algebraically closed field $k$ of characteristic $\Char(k) \geq 0$.
In \cite{Lac}, Lacini gives a classification of log del Pezzo surfaces of rank one when $\Char(k) \neq 2,3$.
The technique is based on the work of Keel and M\textsuperscript{c}Kernan \cite{KM}.
To make their technique applicable in positive characteristic, Lacini overcomes difficulties such as those involving topological arguments and the absence of the Bogomolov bound \cite[\S 9]{KM}.

The classification is carried out based on the existence of tigers in the minimal resolution (see Definition \ref{def:tiger}). 
The case without tigers is handled in \cite[\S 4]{Lac}. 
On the other hand, in the case with tigers, Lacini applies \cite[Theorem 6.2]{Lac} to reduce the classification to that of certain almost log canonical pairs, which is carried out in \cite[\S 5]{Lac}.
In fact, the classification of the case with tigers can be simplified by employing the following modified version of \cite[Theorem 6.2]{Lac}.

\begin{thm}[{\cite[Theorem 5.3]{Nag25}}]\label{Lac6.2}
Let $S$ be a rank one log del Pezzo surface with tigers in its minimal resolution $\wt S$.
\begin{enumerate}
    \item[\textup{(1)}] Suppose that $S$ has an exceptional tiger $E$ in $\wt S$.
    Take $f$ and $\pi$ as in Lemma \ref{lem:Sarkisov}.
    Let $C$ be the support of $\Delta_1$ if $\pi$ is birational.
    Then, by replacing $E$ if necessary, one of the following holds.
    \begin{enumerate}
        \item[\textup{(a)}] $T$ is a net, $E$ is a section or a bisection, and $K_T+E$ is plt.
        \item[\textup{(b)}] $\pi$ is birational, $S_1$ is a Du Val del Pezzo surface of rank one, and $C$ is an irreducible, reduced, and singular anti-canonical member of $S_1$ contained in $S_1^\circ$.
        The image of the $\pi$-exceptional curve in $S_1$ is the singularity of $C$.
        \item[\textup{(c)}] $\pi$ is birational and $S_1$ is a log del Pezzo surface of rank one. 
        $K_{S_1}+C$ is plt and anti-nef.
        The image of the $\pi$-exceptional curve in $S_1$ is contained in $C$.
    \end{enumerate}
    \item[\textup{(2)}] If $S$ has tigers in $\wt S$ but none of which is exceptional, then $S$ 
    contains a reduced and irreducible curve $C$ such that $K_S+C$ is plt and anti-nef. 
\end{enumerate}
\end{thm}

This theorem shows that the classification of certain almost log canonical pairs can be reduced to that of the pairs in assertion (1)(b) or (1)(c). 
The former can be easily classified in a manner similar to \cite[Lemma B.9]{Lac} (see \S \ref{DuVal} for more details), and the latter is classified in \cite[\S 5.1]{Lac}.
In particular, we observe that the results of \cite[\S 5.2]{Lac} need not be used.

For this reason, in this paper, we review \cite[\S\S 4--5.1]{Lac} with a focus on where and how Lacini enhanced the techniques of \cite{KM}.
We also provide modified proofs for several results.
As a consequence, we point out that there is at most one log del Pezzo surface that may have been erroneously omitted from the list in \cite[\S 6.1]{Lac}, namely the surface in Example \ref{eg:cexLac4.19}.
We do not know whether such a surface has a tiger in its minimal resolution.
We also clarify where the assumption $\Char(k) \neq 2,3$ is needed, and extend the results of \cite[\S 4.2]{Lac} to arbitrary characteristic.

This paper is structured as follows.
\S \ref{sec:pre} contains some definitions and preliminary lemmas.
In \S \ref{LacS4} and \S \ref{LacS5.1}, we present the content of \cite[\S 4]{Lac} and \cite[\S 5.1]{Lac} respectively and compare them with \cite{KM}.
In \S \ref{DuVal}, we briefly give the classification of pairs as in assertion (1)(b) of Theorem \ref{Lac6.2} when $\Char(k) \neq 2,3$.

\subsection{Notation}
We work over an algebraically closed field $k$ of characteristic $\Char(k) \geq 0$.
A \textit{variety} means an integral separated scheme of finite type over $k$.
A \textit{curve} (resp.\ a \textit{surface}) means a variety of dimension one (resp.\ two).
A \textit{log pair} $(S, \Delta)$ is the pair of a normal projective variety $S$ and an effective $\Q$-Weil divisor $\Delta$ such that $K_S+\Delta$ is $\Q$-Cartier.
We say that a $\Q$-Weil divisor $\Delta$ is a \textit{boundary} if the irreducible decomposition $\Delta = \sum_i c_i D_i$ satisfies $0 \leq c_i \leq 1$.
We say that a normal surface is \textit{Du Val} if it has only canonical singularities. 
A \textit{log del Pezzo surface} is a projective surface with only klt singularities whose anti-canonical divisor is ample.
We often call the Picard rank of a log del Pezzo surface simply the \textit{rank}.

Throughout this paper, we use the following notation:
\begin{itemize}
\item $\wt S$ : the minimal resolution of a normal surface $S$.
\item $S^\circ$ : the smooth locus of a normal surface $S$.
\item $C_T$ : the strict transformation of a closed subscheme $C$ of a normal variety $S$ in a birational model $T$ of $S$.
\item $g(C)$: the arithmetic genus of a projective curve $C$.
\end{itemize}

\section{Preliminary}\label{sec:pre}

\subsection{Dynkin types}

In this subsection, we set up notation and terminology of klt singularities of dimension two.

\begin{defn}
For integers $c, l_1, \ldots, l_i, m_1, \ldots, m_j, n_1, \ldots n_k \in \Z_{\geq 2}$, the symbol $[n_1, \ldots, n_k]$ (resp.\ $[c; [l_1, \ldots, l_i], [m_1, \ldots, m_j], [n_1, \ldots, n_k]]$) stands for the weighted dual graph as in Figure \ref{fig:chain} (resp.\ Figure \ref{fig:star}).
For example, $[2;[2],[2],[2^{k-3}]]$ and $[2;[2],[2^2],[2^{k-4}]]$ are the exceptional divisors of the Du Val singularities $D_k$ and $E_k$, respectively.
\end{defn}

\begin{figure}[ht] 
		\centering
		\begin{tikzpicture}
		
		\draw[] (0,0)circle(2pt);
		\draw[] (1,0)circle(2pt);
		\draw[] (2,0)circle(2pt);
		\node(a)at(3,0){$\cdots$};
		\draw[] (4,0)circle(2pt);
		
		\draw[thin ](0.2,0)--(0.8,0);
		\draw[thin ](1.2,0)--(1.8,0);
		\draw[thin ](2.2,0)--(2.6,0);
		\draw[thin ](3.4,0)--(3.8,0);

		\node(a)at(0,0.45){$-n_1$};
		\node(a)at(1,0.45){$-n_2$};
		\node(a)at(2,0.45){$-n_3$};
		\node(a)at(4,0.45){$-n_k$};

		\end{tikzpicture}
		\caption{The weighted graph $[n_1, n_2, \ldots, n_k]$.}\label{fig:chain}
		\centering
		\begin{tikzpicture}
		
		\draw[ ] (-3,0)circle(2pt);
		\node(a)at(-2,0){$\cdots$};
		\draw[] (-1,0)circle(2pt);
		\draw[] (0,0)circle(2pt);
		\draw[] (1,0)circle(2pt);
		\node(a)at(2,0){$\cdots$};
		\draw[] (3,0)circle(2pt);
		\draw[] (0,-1)circle(2pt);
		\node(a)at(0,-2){$\vdots$};
		\draw[] (0,-3)circle(2pt);
		
		\draw[thin ](-2.8,0)--(-2.4,0);
		\draw[thin ](-1.6,0)--(-1.2,0);
		\draw[thin ](-0.8,0)--(-0.2,0);
		\draw[thin ](0.2,0)--(0.8,0);
		\draw[thin ](1.2,0)--(1.6,0);
		\draw[thin ](2.4,0)--(2.8,0);
		\draw[thin ](0,-0.2)--(0,-0.8);
		\draw[thin ](0,-1.2)--(0,-1.8);
		\draw[thin ](0,-2.4)--(0,-2.8);

		\node(a)at(-3,0.45){$-n_k$};
		\node(a)at(-1,0.45){$-n_1$};
		\node(a)at(0,0.45){$-c$};
		\node(a)at(1,0.45){$-m_1$};
		\node(a)at(3,0.45){$-m_j$};
		\node(a)at(-0.45,-1){$-l_1$};
		\node(a)at(-0.45,-3){$-l_i$};

		\end{tikzpicture}
		\caption{The weighted graph $[c; [l_1, \ldots, l_i]$, $[m_1, \ldots$, $m_j]$, $[n_1, \ldots, n_k]]$.}\label{fig:star}
	\end{figure}

\begin{defn}\label{Dynkin}
Let $S$ be a surface with klt singularities and $\sigma \colon \wt S \to S$ the minimal resolution.
The \textit{Dynkin type} of $S$, denoted by $\Dyn(S)$, is defined to be the weighted dual graph of the reduced $\sigma$-exceptional divisor.
If $S$ has only Du Val singularities, its type is sometimes indicated by enclosing the singularity type in parentheses after $S$. For example, we write $S(3A_1+D_4)$ if $S$ has exactly three $A_1$-singularities and one $D_4$-singularity.
\end{defn}

\subsection{Hunt steps}

In this subsection, we recall certain Sarkisov links called hunt steps, which are used in \cite{KM} and \cite{Lac}.

\begin{defn}
Let $(X, \Delta)$ be a log pair. 
Let $f \colon Y \to X$ be a birational morphism with $Y$ normal, and let $\Gamma$ be the log pullback of $\Delta$ via $f$, i.e., the $\Q$-Weil divisor defined by $K_Y+\Gamma=f^*(K_X+\Delta)$.
For a prime divisor $E$ on $Y$, we define the \textit{coefficient} $e(E; X, \Delta)$ of $E$ with respect to the pair $(X, \Delta)$ to be the multiplicity of $\Gamma$ along $E$.
\end{defn}

\begin{defn}[{\cite[Definition 3.10]{Lac}}]\label{def:tiger}
Let $(X, \Delta)$ be a log pair.
Let $f \colon Y \to X$ be a birational morphism.
We say $X$ \textit{has a tiger in} $Y$ if there exists an effective $\Q$-Cartier divisor $\alpha$ such that
\begin{enumerate}
    \item $K_X+\Delta+\alpha$ is numerically trivial.
    \item If $\Gamma$ is the log pullback of $\Delta+\alpha$, then there is a divisor $E$ of coefficient at least one in $\Gamma$.
\end{enumerate}
Any such divisor $E$ is called \textit{tiger}.
\end{defn}

The following geometric situation will be common.

\begin{defn}[{\cite[Definition 3.11]{Lac}}]
Let $A$ and $B$ be two rational curves on a klt surface $S$ such that $K_S+A+B$ is divisorially log terminal at any singular point of $S$.
We say that $(S, A+B)$ is a
\begin{enumerate}
    \item \textit{banana}, if $A$ and $B$ meet in exactly two points, and there transversally.
    \item \textit{fence}, if $A$ and $B$ meet in exactly one point, and there transversally.
    \item \textit{tacnode}, if $A$ and $B$ meet in at most two points, there is one point $q \in A \cap B$ such that $A+B$ has a node of genus $g \geq 2$ at $q$, and if there is a second point of intersection, then $A$ and $B$ meet there transversally.
\end{enumerate}
\end{defn}

\begin{defn}[{\cite[Definition 8.0.2]{KM}}]
Let $(X, \Delta)$ be a log pair with $\Delta$ a boundary and $\Delta = \sum_{i=1}^n a_i D_i$ the irreducible decomposition.
We say that $(X, \Delta)$ is \textit{flush} (resp.\ \textit{level}) if $e(E; X, \Delta) < \min\{a_1, \ldots, a_n, 1\}$ (resp.\ $e(E; X, \Delta) \leq \min\{a_1, \ldots, a_n, 1\}$) for all exceptional divisors $E$ over $X$.
\end{defn}

\begin{lem}[{\cite[Definition-Lemma 8.2.5]{KM} and \cite[Lemma 3.12]{Lac}}]\label{lem:Sarkisov}
Let $(S, \Delta)$ be a log pair such that $S$ is a log del Pezzo surface of rank one.
Let $f \colon T \to S$ be an extraction of relative Picard rank one of an irreducible divisor $E$ of the minimal resolution of $S$.
The Mori cone $\overline{\mathrm{NE}}(T)$ has two edges, one of which is generated by $E$.
Let $R$ be the other edge.
Let $x=f(E)$, $\Gamma$ the log pullback of $\Delta$, and $\Gamma_\epsilon = \Gamma + \epsilon E$, where $0 < \epsilon \ll 1$.
Assume that $-(K_S+\Delta)$ is ample.
Then the following hold.
\begin{enumerate}
    \item[\textup{(1)}] $R$ is $K_T$-negative and contractible,
    hence there is a rational curve $\Sigma$ that generates $R$.
    Let $\pi$ be the contraction morphism of $R$.
    \item[\textup{(2)}] $K_T+\Gamma_\epsilon$ is anti-ample.
    \item[\textup{(3)}] $\Gamma_\epsilon$ is $E$-negative.
    \item[\textup{(4)}] There is a unique rational number $\lambda$ such that with $\Gamma' = \lambda\Gamma_\epsilon$, $K_T+\Gamma'$ is $R$-trivial. 
    Moreover, $\lambda >1$.
    \item[\textup{(5)}] $K_T+\Gamma'$ is $E$-negative.
    \item[\textup{(6)}] $\pi$ is either a $\PP^1_k$-fibration (called \textit{``net''}) or birational.
    \item[\textup{(7)}] If $\pi \colon T \to S_1$ is birational, and $\Delta_1 = \pi_*\Gamma'$, then $K_{S_1}+\Delta_1$ is anti-ample and $S_1$ is a log del Pezzo surface of rank one.
    \item[\textup{(8)}] If  $(S, \Delta)$ does not have a tiger in a surface $Y$ that dominates $T$, then neither do $(T, \Gamma')$ and $(S_1, \Delta_1)$.  
\end{enumerate}
\end{lem}

\begin{defn}[{\cite[Definition-Remark 8.2.8]{KM} and \cite[Definition 3.13]{Lac}}]
We call the transformation $(f, \pi)$ as in Lemma \ref{lem:Sarkisov} \textit{a hunt step} for $(S, \Delta)$ if $e(E; S, \Delta)$ is maximal among exceptional divisors of the minimal resolution of $S$.
If $x$ is a chain singularity (resp.\ a non-chain singularity) and $\Delta = \emptyset$, we require $E$ not to be a $(-2)$-curve (resp.\ to be the central curve). 
This is always possible by \cite[8.3.9 Lemma and 10.11 Lemma]{KM}.
\end{defn}

We fix the notation that we will use when running the hunt.

\begin{nota}[{\cite[Notation 3.14]{Lac}}]\label{nota:hunt}
    We always start from a surface without boundary, so $\Delta_0=\emptyset$.
    We index by $(f_i, \pi_{i+1})$ the next hunt step for $(S_i, \Delta_i)$. Define:
    \begin{itemize}
        \item $x_i = f_i(E_{i+1}) \in S_i$.
        \item $q_{i+1} =\pi_{i+1}(\Sigma_{i+1}) \in S_{i+1}$.
        \item $\Gamma_{i+1}$ to be the log pullback of $\Delta_i$: $K_{T_{i+1}}+\Gamma_{i+1}=f^*_{i}(K_{S_i}+\Delta_i)$.
        \item $\Delta_{i+1}=\pi_{i+1}(\Gamma'_{i+1})$: it satisfies $K_{T_{i+1}}+\Gamma'_{i+1}=\pi^*_{i+1}(K_{S_{i+1}}+\Delta_{i+1})$.
        \item $A_1 = \pi_1(E_1) \subset S_1$ and $B_2 = \pi_2(E_2) \subset S_2$.
        \item $A_2$ the strict transform of $A_1$ in $S_2$.
    \end{itemize}
    Let $a_1$, $b_2$ be the coefficients of $A_1$, $B_2$ in $\Delta_1$, $\Delta_2$ (which are also the coefficients of $E_1$, $E_2$ in $\Gamma'_1$, $\Gamma'_2$) and $a_2$ the coefficient of $A_2$ in $\Delta_2$.
    We remark that $a_1, b_2 < a_2$ by the flush condition and the previous scaling.
    Let $e_i$ be the coefficient of $E_{i+1}$ in $(T_{i+1}, \Gamma_{i+1})$.
    This is also the coefficient over the pair $(S_i, \Delta_i)$.
    Finally, we indicate by $\overline \Sigma_i$ the image of $\Gamma_i$ in $S_0$ or $S_1$, depending on the context.
\end{nota}

The following proposition is the key to classifying log del Pezzo surfaces of rank one without tigers.

\begin{prop}[{\cite[Proposition 8.4.7]{KM} and \cite[Proposition 3.15]{Lac}}]\label{general-strategy}
Suppose that $S$ is a rank one log del Pezzo surface that has no tigers in $\wt S$.
For the first hunt step: $K_{T_1}+E_1$ is divisorially log terminal. $K_{T_1}+\Gamma_1'$ is flush and one of the following holds.

\begin{enumerate}
    \item[\textup{(1)}] $T_1$ is a net.
\end{enumerate}

Otherwise $K_{S_1}+a_1A_1$ is flush and one of the following holds.

\begin{enumerate}
    \item[\textup{(2)}] $g(A_1)>1$.
    \item[\textup{(3)}] $g(A_1)=1$ and $A_1$ has an ordinary node at $q=q_1$.
    \item[\textup{(4)}] $g(A_1)=1$ and $A_1$ has an ordinary cusp at $q=q_1$.
    \item[\textup{(5)}] $g(A_1)=0$ and $K_{S_1}+A_1$ is divisorially log terminal.
\end{enumerate}

For the second hunt step one of the following holds.

\begin{enumerate}
    \item[\textup{(6)}] $T_2$ is a net.
    \item[\textup{(7)}] $A_1$ is contracted by $\pi_2$, $K_{T_2}+\Gamma'_2$ is flush; $K_{S_1}+A_1$ is divisorially log terminal, $q_2$ is a smooth point of $S_2$, $B_2$ is singular at $q_2$ with a unibranch singularity, and $K_{S_2}+\Delta_2$ is flush away from $q_2$, but is not level at $q_2$. $\Sigma_2$ is the only exceptional divisor at which $K_{S_2}+\Delta_2$ fails to be flush.
    \item[\textup{(8)}] $\Delta_2$ has two components.
\end{enumerate}

Suppose that in this last case that $a_2+b_2 \geq 1$.
Then $\overline \Sigma_2 \cap \Sing(A_1) = \emptyset$, $K_{T_2}+\Gamma'_2$ is flush away from $\Sing(A_1)$.
$K_{S_2}+\Delta_2$ is flush away from $\pi_2(\Sing(A_1))$, and at least one of $-(K_{S_2}+A)$ or $-(K_{S_2}+B)$ is ample.
Also, one of the following holds:

\begin{enumerate}
    \item[\textup{(9)}] $(S_2, A+B)$ is a fence.
    \item[\textup{(10)}] $(S_2, A+B)$ is a banana, $K_{S_2}+B$ is plt, and $x_1 \in A$.
    \item[\textup{(11)}] $(S_2, A+B)$ is a tacnode, with tacnode at $q_2$.
    $K_{S_2}+B$ is plt. 
    If $x_1 \in A$, $A \cap B = \{x_1,q_2\}$.
    If $x_1 \not \in A$, then $A \cap B = \{q_2\}$.
\end{enumerate}    
\end{prop}

\subsection{Klt singularities with small coefficients}

In this subsection, we recall the classification for klt surface singularities with coefficient $<\frac35$.

\begin{defn}
For a klt surface germ $(S, x)$, let $\wt S \to S$ be the minimal resolution and $\sum_{i=1}^n E_i$ the exceptional divisors.
Then we define \textit{the gap of $(S, x)$} as $\mathrm{Gap}(x) \coloneqq n+\sum_{i=1}^n e(E_i, K_S)(2+E_i^2)$.
\end{defn}

\begin{lem}[cf. {\cite[Lemma 3.30]{Lac}}]\label{lem:gap}
Let $(S, x)$ be a klt surface germ with $e(x)<\frac35$.
Then Table \ref{tab:e35} is the list of $\mathrm{Gap}(x)$ and its integral part.
\begin{table}[htbp]
    \centering
\caption{Gaps of klt singularities with coefficient $< \frac35$}
    \label{tab:e35}
    \renewcommand{\arraystretch}{1.5}
\begin{tabular}{|c||c|c|}\hline
 $\Dyn(x)$ &$\mathrm{Gap}(x_0)$& $\lfloor \mathrm{Gap}(x_0) \rfloor$  \\ \hline \hline
 \text{Du Val singularity of index $k$} & $k$ & $k$ \\ \hline 
 $[3, 2^k]$ with $k \in \ZZ_{\geq 0}$ & $\frac{2(k+1)^2}{2k+3}$ & $k$  \\ \hline
 $[4]$ & $0$ & $0$ \\ \hline
 $[3, 2^k, 3]$ with $k \in \ZZ_{\geq 0}$ & $k+1$ & $k+1$ \\ \hline
 $[2,3,2]$   &$\frac52$       & $2$\\ \hline
 $[2,3,2^2]$ &$\frac{38}{11}$ & $3$\\ \hline
 $[2,3,2^3]$ &$\frac{31}{7}$  & $4$\\ \hline
 $[2,3,2^4]$ &$\frac{92}{17}$ & $5$\\ \hline
 $[2; [2], [2], [2^k, 3]]$& $\frac{2k+7}{2}$ & $k+3$\\ \hline
 $[2,4]$ & $\frac67$ & $0$ \\ \hline
\end{tabular}
\end{table}

\end{lem}

\begin{proof}
The assertion follows from \cite[Proposition 10.1]{KM} and easy computations.
\end{proof}

\begin{lem}[{\cite[Lemma 3.28]{Lac}}]\label{lem:totalgap}
Let $S$ be a log del Pezzo surface of rank one such that $\Sing(S)=\{x_1, \ldots, x_m\}$.
Then 
\begin{align*}
    K_S^2+\sum_{i=1}^m \mathrm{Gap}(x_i)=9.
\end{align*}
If $e(S)<\frac35$ in addition, then 
\begin{align*}
    K_S^2+\sum_{i \in I} \mathrm{Gap}(x_i) \leq 9
\end{align*}
for any $I \subset \{1, \ldots, m\}$.
\end{lem}

\begin{proof}
The last assertion follows from Lemma \ref{lem:gap}.
\end{proof}

\subsection{Special configurations in $\PP^2_k$}

The next lemma is important for classifying log del Pezzo surfaces in \S \ref{LacS4.2}.

\begin{lem}[{\cite[Lemma 4.1]{Nag25}}]\label{lem:specialconfig}
In $\PP^2_k$, let $C$ be a cuspidal curve and $Q$ a smooth conic intersecting with $C$ at a smooth point $t$ of $C$ with multiplicity at least five.
Take a point $s$ so that $C \cap Q=\{s, t\}$.
(We admit that $s=t$.)
Let $u$ be the cusp of $C$.
Take the line $M_u$ which intersects with $C$ at $u$ with multiplicity three and the line $L_{su}$ (resp.\ $L_{tu}$) passing through $s$ (resp.\ $t$) and $u$.
Then the following hold:
\begin{enumerate}
    \item[\textup{(1)}] When $\Char(k) \neq 3$, then there are coordinates $[x:y:z]$ of $\PP^2_k$ such that $C=\{x^3=y^2z\}$, $Q = \{-45x^2-5y^2+z^2+24xy+40yz-15zx=0\}$, $M_u=\{y=0\}$, and $t=[1:1:1]$.
    \item[\textup{(2)}] When $\Char(k) = 3$, then there are coordinates $[x:y:z]$ of $\PP^2_k$ such that $C = \{x^3=y^2z+x^2y\}$, $Q = \{-x^2+z^2-yz-zx=0\}$, $M_u=\{y=0\}$, and $t=[0:1:0]$.
    \item[\textup{(3)}] $Q \cap L_{tu}$ consists of exactly two points.
    \item[\textup{(4)}] $s=t$ if and only if $\Char(k)=2$.
    \item[\textup{(5)}] $M_u$ is the tangent line of $Q$ at some point if and only if $\Char(k)=5$.
    \item[\textup{(6)}] $L_{su}$ is the tangent line of $Q$ at $s$ if and only if $\Char(k)=5$.
\end{enumerate}
\end{lem}

\section{Log del Pezzo surfaces without tigers}\label{LacS4}

In \cite[\S 4]{Lac}, Lacini classified log del Pezzo surfaces without tigers in the minimal resolution by running hunts.
In the remainder of this section, we assume that $S_0$ is a log del Pezzo surface of rank one with no tigers in $\wt S_0$.
Then Proposition \ref{general-strategy} shows that one of the following holds:
\begin{enumerate}
    \item $T_1$ is a net.
    \item $g(A_1) >1$.
    \item $g(A_1) =1$ and $A_1$ has a simple node at $q=q_1$.
    \item $g(A_1) =1$ and $A_1$ has a simple cusp at $q=q_1$.
    \item $g(A_1) =0$ and $K_{S_1}+A_1$ is divisorially log terminal.
\end{enumerate}
By Lemma \ref{lem:Sarkisov} (8), $(S_1, a_1A_1)$ does not have a tiger in $\wt S_0$.
Cases (1)--(5) are handled in \cite[\S\S 4.1--4.4 and 4.6]{Lac}, respectively.

In the remainder of this section, we use the content of \cite[\S 3.4]{Lac} to describe the contractions $\pi_1$ and $\pi_2$.

\subsection{$T_1$ is a net}\label{LacS4.1}

\cite[\S 4.1]{Lac} is devoted to the case where $T_1$ is a net.
In $\Char(k)=0$, such a log del Pezzo surface $S_0$ is classified in \cite[\S 14]{KM} under the additional assumption that $\pi_1^{\mathrm{alg}}(S_0^\circ) =0$.
This assumption is used to obtain \cite[Lemma 14.3]{KM}.
Furthermore, the proof of \cite[Lemma 14.4]{KM} uses \cite[Lemma 10.8]{KM}, which relies on the Bogomolov bound \cite[\S 9]{KM}.

For this reason, in \cite[\S 4.1]{Lac}, Lacini avoids the use of $\pi_1^{\mathrm{alg}}(S_0^\circ)$ and the Bogomolov bound by checking all of the possible singular fibers, utilizing \cite[Lemmas 3.25, 3.26, and 3.29]{Lac} (or \cite[Lemma 11.5.9, 11.5.13, and (10.3)]{KM}).
The assumption $\Char(k) \neq 2,3$ is only used to apply the Riemann-Hurwitz formula to $E_1$ when it is a $3$-section.

\subsection{$g(A_1) >1 $}\label{LacS4.2}

\cite[\S 4.2]{Lac} is devoted to the case where $g(A_1)>1$.
In $\Char(k)=0$, such a log del Pezzo surface $S_0$ is classified in \cite[\S 15]{KM}.
In this argument, we frequently use \cite[Lemma 10.8]{KM}, which relies on the Bogomolov bound.
For this reason, Lacini provides a classification method in \cite[\S 4.2]{Lac} that is independent of the Bogomolov bound.

The classification method of Lacini is as follows.
Since $S_0$ has no tigers in $\wt S_0$, it is not Du Val and hence $\frac13 \leq e_0 <a_1$.
Since $(S_1, a_1A_1)$ is flush, 
\cite[Lemma 8.3.7]{KM} shows that $\mult_{q_1}A_1=2$ or $3$.
Since $g(A_1)>1$, there are finitely many possibilities for the configuration of $\pi_1$, which are listed in \cite[\S 11.1-3]{KM}.
On the other hand, by Lemma \ref{lem:totalgap}, we obtain
\begin{align} \label{eq:dioph1}
    K_{S_0}^2 + \sum_{x \in \Sing(S_0)} \Gap(x)= 9.
\end{align}
Since $S_0$ is of rank one, we obtain
\begin{align}\label{eq:dioph2}
    K_{S_0}^2 = \dfrac{(K_{S_0} \cdot \overline{\Sigma}_1)^2}{\overline{\Sigma}_1^2},
\end{align}
where $\overline{\Sigma}_1 = \Sigma_{1,S_0}$. 
For each configuration, we can calculate $\overline{\Sigma}_1^2$ and $(K_{S_0} \cdot \overline{\Sigma}_1)$ from the Dynkin type of $x_0$.
Since $x_0$ is one of the singularities with maximum coefficient among $\Dyn(S_0)$, equations (\ref{eq:dioph1}) and (\ref{eq:dioph2}) give us 
\begin{align}\label{eq:dioph3}
     \dfrac{(K_{S_0} \cdot \overline{\Sigma}_1)^2}{\overline{\Sigma}_1^2} + \sum_{x \in \Sing(S_0)} \Gap(x)= 9,
\end{align}
which $\Dyn(S_0)$ should satisfy.

Suppose in addition that $\pi_1$ is not in configuration $u$ or $v$ as in \cite[Lemma 11.2.1]{KM}. Then $e_0 < a_1 \leq \frac35$. 
By Lemma \ref{lem:gap}, $\Dyn(S_0)$ consists of singularities as listed in Table \ref{tab:e35}.
In particular, each $\Gap(x)$ is non-negative and is zero only when $x=[4]$.
Therefore, by virtue of equation (\ref{eq:dioph3}), we can narrow down the possibilities for $\Dyn(S_0)$ to finitely many (disregarding the possibility of [4]-singularities being added).
At this stage, we use computer assistance.
As a result, we obtain $g(A_1)=2$.
We note that $g(A_1)=2$ even when $\pi_1$ is in configuration $u$ or $v$ by \cite[Lemma 11.2.1]{KM}.

Finally, we consider the second hunt step or the $\PP^1_k$-fibration $\wt S_1 \to \PP^1_k$ defined by
$\wt F$ as in \cite[12.3]{KM} to determine the construction of $S_0$.

We take a similar approach to that mentioned above for the case where $g(A_1)=1$ and $\pi_1$ is in configuration I, but we determine the construction via the birational morphisms $S_1 \leftarrow Y \to W$ as in \cite[12.4]{KM} instead.

The assumption $\Char(k) \neq 2,3$ is used only in the following. 
We need $\Char(k) \neq 2$ to apply \cite[Lemma 15.3]{KM}. 
We apply the Riemann-Hurwitz formula to $2$-sections obtained as in \cite[12.3]{KM}. 
We also check whether the configurations of curves in the output of hunts are feasible in the given characteristic.

In the remainder of this subsection, we will make the content of \cite[\S 4.2]{Lac} characteristic-free.
As a result, we will establish the following:
\begin{prop}[{\cite[Proposition 4.4]{Lac}}]\label{Lac4.5}
Suppose that $g(A_1)>1$. 
Then one of the following holds.
\begin{enumerate}
    \item[\textup{(1)}] $\Char(k)=2$ and $S_0$ is constructed as in Example \ref{eg:Lac4.7char2}.
    \item[\textup{(2)}] $\Char(k) \geq 0$ and $S_0$ is constructed as in Example \ref{eg:Lac4.7charany}. Such an $S_0$ is also constructed as in \cite[Proposition 4.4 (1)]{Lac}.
    \item[\textup{(3)}] $\Char(k)=5$ and $S_0$ is constructed as in Example \ref{eg:Lac4.8cex-1}. 
    Such an $S_0$ is also constructed as in \cite[Proposition 4.4 (2) and (3)]{Lac}.
\end{enumerate}

\end{prop}

In the following proof, we have endeavored to write out the parts referred to as "tedious computations" in \cite{Lac} with as little omission as possible, and whenever we employ a computer method, we explicitly verify that the candidate values are already finite.
Conversely, for parts that are adequately detailed in [\textit{ibid}.], we refer to them.

In what follows, we may assume that $A_1$ has a double point since the proof of \cite[Lemma 4.5]{Lac} is already characteristic-free. 
Since $E_1$ is smooth, $\pi_1$ is not in configuration $0$.
We begin with configuration I, where we also allow $g(A_1)=1$ following \cite{Lac}.

\begin{lem}[{\cite[Lemma 4.6]{Lac}}]\label{Lac4.6}
Suppose $A_1$ has a double point and $g(A_1) \geq 1$ and that the configuration I arises.
Then $\Char (k) = 3$, $g(A_1)=1$, and $S_0$ is constructed as in Example \ref{eg:cexLac4.6}. 
\end{lem}

\begin{eg}\label{eg:cexLac4.6}
In $\Char (k) =3$, let $Z$ be a quasi-elliptic surface of type (2) or (3) as in \cite[Theorem 3.1]{Ito1}.
Let $A' \subset Z$ be a general fiber, $s \subset Z$ a section, and $C' \subset Z$ the irreducible component of a singular fiber of type $IV$ intersecting with $s$.
Let $Z \to W$ be the contraction of $s$ and let $A$ (resp.\ $C$) be the image of $A'$ (resp.\ $C'$) in $W$.
Note that $C$ intersects with exactly two $(-2)$-curves, and these three curves share a unique point, say $t$.

Next, blow-up $W$ at $t$.
Finally, take a blow-up at the cusp of $A$.
Then we obtain the minimal resolution $\wt S$ of a log del Pezzo surface $S$ of rank one such that $\Dyn(S)=[2, 3]+2[3]+[2;[2],[2^2],[2^2]]$ (resp.\ $[2,3]+2[3]+3[2^2]$) when $Z$ is of type $(2)$ (resp.\ $(3)$).
Indeed, in each case, the strict transform of the $(-1)$-curve over the cusp in $S$ is of $(-K_S)$-degree $\frac{1}{5}$ and hence $-K_S$ is ample.
\end{eg}

\begin{proof}[Proof of Lemma \ref{Lac4.6}]
As in the proof of \cite[Lemma 4.6]{Lac}, we obtain $x_0=[3, 2^r]$ for some $r \geq 0$.
Since the Picard rank of $S_0$ is one, we have 
\begin{align*}
    K_{S_0}^2=\dfrac{(K_{S_0} \cdot \overline{\Sigma}_1)^2}{\overline{\Sigma}_1^2}
    =\dfrac{(-1+2e_0)^2}{4e_0-1/g}=\dfrac{g}{(2r+3)(4(r+1)-g(2r+3))},
\end{align*}
where $g=g(A_1)$. Then equation (\ref{eq:dioph3}) gives us

\begin{align}\label{eq:diophantine1}
    \dfrac{g}{(2r+3)(4(r+1)-g(2r+3))}+v+\sum_{k=0}^r n_k \cdot \dfrac{2(k+1)^2}{2k+3}=9,
\end{align}
where $v$ is the number of exceptional curves coming from Du Val singularities, and $n_k$ is the number of points of type $[3,2^k]$. 
From this formula, we obtain that 
$r < 9$, $v < 9$, $3v+2n_0 < 27$, and $v+\sum_{k=0}^r n_k \cdot k <9$. 
By the maximality of $r$, we have $n_k = 0$ for $k > r$.
By the shape of configuration I, there is a $[2^{g-1}]$-singularity on $\overline{\Sigma}_1$. 
In particular, we have $1 \leq g \leq v+1$.
Hence there are only finitely many possibilities for the tuple $(g, r, v, n_0, \ldots, n_8)$.

A computer search shows that (\ref{eq:diophantine1}) requires that $g=1$ and $n_k=0$ for $k\geq 3$. 
The resulting solutions $(r,v, n_0,n_1, n_2)$ are listed below:
\begin{align*}
(r, v, n_0, n_1, n_2)=
\begin{cases}
(0,2m,13-3m,0,0) & (0 \leq m \leq 4),\\
(1,2m,11-3m,1,0) &(0 \leq m \leq 3), \\
(2,0,0,4,1).  &
\end{cases}
\end{align*}

Assume in addition that $r=0$.
Then $A_1$ is a curve of genus one contained in the smooth locus of $S_1$.
Hence $A_1$ is a tiger of $(S_1, a_1A_1)$, a contradiction.
Thus $r>0$.

Now we use the notation of \cite[Lemma-Definition 12.4]{KM}.
Let $f \colon Y \to S_1$ be the extraction of the exceptional divisor $G$ of $\wt S_1$ adjacent to $A_1$ and let $\pi \colon Y \to W$ be the induced morphism. 
As in the proof of \cite[Lemma 4.6]{Lac}, we obtain the following:
\begin{enumerate}
    \item $(r, v, n_0, n_1, n_2)=(1, 6, 2, 1, 0)$.
    \item $W$ is a Du Val del Pezzo surface with $K_W^2=1$.
    \item The center of $\pi$ is a $[2^2]$-singularity.
    \item $C \coloneqq G_W$ passes through the intersection of the $[2^2]$-singularity.
\end{enumerate}
Set $A \coloneqq A_{1,W}$ and let $Z$ be the minimal resolution of the blow-up of $W$ at $\Bs |-K_W|$, which is an extremal rational (quasi-)elliptic surface.
If $\Bs |-K_W| \neq A \cap C$, then the strict transform of $C$ in $Z$ is a section passing through the intersection of two $(-2)$-curves, a contradiction.
Hence $\Bs |-K_W| = A \cap C$ and $Z$ is an extremal rational (quasi-)elliptic surface which has a $IV$ and either an $I_1$ or $II$ among its singular fibers.
By the classification of the extremal rational (quasi-)elliptic surfaces \cite{M-P, Lang1, Lang2, Ito1, Ito2}, 
we conclude that $\Char(k)=3$, $Z$ is the extremal rational quasi-elliptic surface of type (2) or (3) as in \cite[Theorem 3.1]{Ito1}, $A_Z$ is its general fiber, and $C_Z$ is an irreducible component of a reducible fiber of type $IV$. 
Therefore, $S$ is constructed as in Example \ref{eg:cexLac4.6}.
\end{proof}

\begin{lem}[{\cite[Lemma 4.7]{Lac}}]\label{Lac4.7}
Suppose $A_1$ has a double point, $g(A_1) \geq 2$, and that the configuration II arises.
Then one of the following holds.
\begin{enumerate}
    \item[\textup{(1)}] $\Char(k)=2$ and $S_0$ is constructed as in Example \ref{eg:Lac4.7char2}.
    \item[\textup{(2)}] $\Char(k) \geq 0$ and $S_0$ is constructed as in Example \ref{eg:Lac4.7charany}.
    In $\Char(k) = 0$, it is also constructed as in \cite[Lemma 15.2]{KM}.
\end{enumerate}
\end{lem}

\begin{eg}\label{eg:Lac4.7char2}
In $\Char(k)=2$, we follow the notation as in Lemma \ref{lem:specialconfig}.
Next, blow-up at $t$ six times along $C$.
Next, blow-up at $u$ three times along $C$.
Finally, blow-up at one of the two points $Q \cap M_u$.
Then we obtain the minimal resolution $\wt S$ of a log del Pezzo surface $S$ of rank one with $\Dyn(S)=[3, 2^2]+2[3]+[2^5]$.
Indeed, the strict transform of the exceptional divisor of the last blow-up in $S$ is of $(-K_S)$-degree $\frac{2}{7}$ and hence $-K_S$ is ample.

Note that in the first hunt, the configuration II occurs and $A_1 \subset S_1$ is the strict transform of $Q$, which has a nodal double point and $g(A_1)=2$.
By \cite[Lemma 11.1.1]{KM}, additional blow-ups at points over $q_1 \in S_1$ do not give another log del Pezzo surface with no tigers.
\end{eg}

\begin{eg}\label{eg:Lac4.7charany}
In $\Char(k) \geq 0$, we follow the notation as in Lemma \ref{lem:specialconfig}.
Take $v \neq t$ such that $L_{tu} \cap Q=\{t, v\}$.
Next, blow-up at $t$ five times along $C$.
Next, blow-up at $u$ three times along $C$.
Next, blow-up at $v$ once.
Finally, blow-up at one of the points $Q \cap M_u$ once.
Then we obtain the minimal resolution $\wt S$ of a log del Pezzo surface $S$ of rank one with $\Dyn(S)=[2,3,2^2]+[3,2^5]$.
Indeed, the strict transform of the exceptional divisor of the last blow-up in $S$ is of $(-K_S)$-degree $\frac{1}{11}$ and hence $-K_S$ is ample.

Note that in the first hunt, the configuration II occurs and $A_1 \subset S_1$ is the strict transform of $Q$, which has a double point and $g(A_1)=2$.
Hence, in $\Char(k) =0$, $S_0$ is the same as the log del Pezzo surface constructed as in \cite[Lemma 15.2]{KM}.
On the other hand, additional blow-ups at points over $q_1 \in S_1$ give another log del Pezzo surface with no tigers only if $A_1$ has a cusp by \cite[Lemma 11.1.1]{KM} and hence only if $\Char(k)=5$ by Lemma \ref{lem:specialconfig} (5).
In fact, we can check that by blowing-up at the intersection of the strict transforms of $Q$ and $M_u$ in $\wt S$, we obtain a log del Pezzo surface constructed as in Example \ref{eg:Lac4.8cex-1} with $(n_t, n_u)=(6,3)$.
\end{eg}

\begin{proof}[Proof of Lemma \ref{Lac4.7}]
We write $g=g(A_1)$. As in the proof of \cite[Lemma 4.7]{Lac}, we obtain $x_0=[2,3,2,2]$ with $g=2$ or $x_0=[3, 2^g]$.

\noindent\underline{\textbf{Case 1.}}
Suppose that $x_0=[3, 2^g]$.
Since the Picard rank of $S_0$ is one, we have 
\begin{align*}
    K_{S_0}^2=\dfrac{(K_{S_0} \cdot \overline{\Sigma}_1)^2}{\overline{\Sigma}_1^2}
    =\dfrac{(-1+\frac{2g+1}{2g+3})^2}{\frac{6g+1}{2g+3}-1}=\dfrac{2}{(2g+3)(2g-1)}.
\end{align*}
Then equation (\ref{eq:dioph3}) gives us

\begin{align}\label{eq:diophantine2}
    \dfrac{2}{(2g+3)(2g-1)}+v+\sum_{k=0}^g n_k \cdot \dfrac{2(k+1)^2}{2k+3}=9,
\end{align}
where $v$ is the number of exceptional curves coming from Du Val singularities, and $n_k$ is the number of points of type $[3,2^k]$. 
From this formula, we obtain that $v < 9$, $3v+2n_0 < 27$, and $v+\sum_{k=0}^gn_k \cdot k <9$. 
In particular, $g < 9$.
By the maximality of $g$, we have $n_k = 0$ for $k > g$.
Hence there are only finitely many possibilities for the tuple $(g, v, n_0, \ldots, n_8)$.
A computer search shows that (\ref{eq:diophantine2}) requires that $g=2$, $n_k=0$ for $k\geq 3$, and $(v,n_0,n_1, n_2) = (1+2m, 8-3m, 0, 1)$ with $0 \leq m \leq 2$.

The rest of the proof is the same as Case 1 in the proof of \cite[Lemma 4.7]{Lac}.
Indeed, we obtain $(v, n_0)=(5, 2)$.
We have $A_1^2=6$ and $\Sing(S_1)$ consists of $2[3]$-singularities and Du Val singularities of total index five.
Let $E_1$ and $E_2$ be the exceptional divisors over the $2[3]$-singularities and let $L$ be the middle $(-2)$-curve in $\wt S_0$ over the $[3,2^2]$-singularity.
Subsequently, we can construct a birational morphism $f \colon \wt S_0 \to \PP^2_k$ such that $f(E_2)$ is a cuspidal cubic, 
$f(A_1)$ is a conic intersecting with $f(E_2)$ at one point with multiplicity six,
and $f(L)$ is a line intersecting with $f(E_2)$ only at its cusp.
By Lemma \ref{lem:specialconfig}, we conclude that $\Char(k)=2$ and $S_0$ is constructed as in Example \ref{eg:Lac4.7char2}.

Note that $f(L)$ is tangent to $f(A_1)$ if and only if $A_1$ is cuspidal.
By Lemma \ref{lem:specialconfig} (5), we conclude that $f(L)$ is not tangent to $f(A_1)$.

\noindent\underline{\textbf{Case 2.}}
Suppose that $x_0=[2,3,2^2]$.
Since the Picard rank of $S_0$ is one, we have 
\begin{align*}
    K_{S_0}^2=\dfrac{(K_{S_0} \cdot \overline{\Sigma}_1)^2}{\overline{\Sigma}_1^2}
    =\dfrac{(\frac{1}{11})^2}{\frac{13}{11}}=\dfrac{1}{11 \cdot 13}.
\end{align*}
Since $e_0=\frac{6}{11}<\frac35$, Lemma \ref{lem:gap} shows that $S_0$ must have a $[3,2^5]$-singularity to compensate for the factor 13 in the denominator.
Since $K_{S_0}^2+\Gap([2,3,2^2])+\Gap([3,2^5])=9$, Lemma \ref{lem:totalgap} shows that $\Dyn(S_0)=[2,3,2^2]+[3,2^5]+n[4]$ for some $n \geq 0$.
In particular, $\Dyn(S_1)=[2]+[3,2^5]+n[4]$ and the $[2]$-singularity lies on $A_1$.
Since $a_1>\frac{6}{11}>\frac12$, we can take $F+M \in |K_{S_1}+A_1|$ as in \cite[\S 12]{KM}.

Assume that $n \geq 1$.
Then, by \cite[Lemma 12.1 (4) and 12.3 (3)]{KM}, $M$ passes through all the $n[4]$-singularities.
By \cite[Lemma 12.1 (7)]{KM}, $M$ is irreducible.
By \cite[Lemma 12.1 (1)]{KM}, $K_{S_1}+M$ is negative and the reduced structure of $M$ is isomorphic to $\PP^1_k$.
In particular, for each singularity on $M$, there is exactly one exceptional divisor of $\wt S_1$ which intersects with $M_{\wt S_1}$ transversally. 

If $n \geq 3$, then $(K_{S_1}+M \cdot M) \geq 3 \cdot \frac34>0$, a contradiction.
If $n = 2$, then the negativity of $K_{S_1}+M$ shows that $M$ contains exactly $2[4]$-singularities, which leads to the contradiction that $M$ is contractible.
If $n = 1$, then the negativity of $K_{S_1}+M$ shows that $M$ contains at most one additional singularity other than the $[4]$-singularity, which again leads to the contradiction that $M$ is contractible.

Thus $n=0$.
Now one can determine the configuration of some negative curves as in \cite[Definition-Lemma 15.2]{KM}.
We follow the notation of [\textit{ibid}.].
Let $H$ be the marked curve $(3^B, 2^H, A_3, 2^L_{ac})$.
Contracting $\Sigma$, $G$, $F_1$, $F_2$, and all the exceptional curves of the resolution $\wt S_0 \to S_0$ other than $M_b$, $E_1$, $L_{bb'}$, and $H$, we obtain a birational map $f \colon \wt S_0 \to \PP^2_k$.
Then $f(M_b)$ is a cuspidal cubic singular at $f(B)$,
$f(A)$ is a conic intersecting with $f(M_b)$ at $f(F_1)$ with multiplicity at least five,
$f(L_{bb'})$ is a line intersecting with $f(M_b)$ only at the cusp,
and $f(H)$ is the line passing through $f(F_1)$ and the cusp.
By Lemma \ref{lem:specialconfig}, we conclude that $S_0$ is constructed as in Example \ref{eg:Lac4.7charany}.
\end{proof}

\begin{lem}[{\cite[Lemma 4.8]{Lac}}]\label{Lac4.8}
Suppose $A_1$ has a double point, $g(A_1) \geq 2$, and that the configuration III arises.
Then $\Char(k) = 5$ and $S_0$ is constructed as in Example \ref{eg:Lac4.8cex-1} with $(n_t, n_u)=(5,3)$ or $(6,3)$.
\end{lem}

\begin{eg}\label{eg:Lac4.8cex-1}
In $\Char(k) = 5$, we follow the notation as in Lemma \ref{lem:specialconfig}.
Set $(n_t, n_u)=(5,3)$, $(6, 3)$, or $(5, 4)$.
Next, blow-up at $s$ twice along $L_{su}$.
Next, blow-up at $t$ $n_t$ times along $C$.
Finally, blow-up at $u$ $n_u$ times along $C$. 
Then we obtain the minimal resolution $\wt S$ of a log del Pezzo surface $S$ of rank one.
Indeed, we have the following:
\begin{align*}
\Dyn(S)&=
\begin{cases}
    2[3,2]+[3]+[2]+[2^4]      & ((n_t, n_u)=(5,3))\\
    [4,2]+[3,2^5]+[3,2]+[2]   & ((n_t, n_u)=(6, 3))\\
    [2,4]+[2,3,2^2]+[3]+[2^4] & ((n_t, n_u)=(5, 4)).
\end{cases}\\
(-K_S \cdot L)&=
\begin{cases}
    \frac15      & ((n_t, n_u)=(5,3))\\
    \frac{1}{35} & ((n_t, n_u)=(6, 3))\\
    \frac{5}{77} & ((n_t, n_u)=(5, 4)),
\end{cases}
\end{align*}
where $L$ is the strict transform of the $(-1)$-curve over $u$.
\end{eg}

\begin{proof}[Proof of Lemma \ref{Lac4.8}]
We write $g=g(A_1)$.
By \cite[Lemma 11.2.1]{KM}, we have $e_0 < a_1=\frac{g+1}{2g+1} \leq \frac35$. 
Thus the possibilities for $x_0$ are listed in Table \ref{tab:e35}.
By equation (\ref{eq:dioph2}), we can calculate $K_{S_0}^2$ as in Table \ref{tab:K2-Lac4.8}.
In what follows, we will treat all the candidates for $x_0$ to conclude that $x_0=[4,2]$ or $[3,2]$. 
\begin{table}[htbp]
\caption{The value of $K_{S_0}^2$ in the setting of Lemma \ref{Lac4.8}}
    \label{tab:K2-Lac4.8}
    \renewcommand{\arraystretch}{2}
\begin{tabular}{|c||c|}\hline
 $x_0$ &$K_{S_0}^2$ \\ \hline \hline
 $[3, 2^k]$ with $k \in \ZZ_{\geq 0}$ & $\dfrac{2(k+g+2)^2}{(2k+3)(2g+1)(4gk+4g-1)}$  \\ \hline
 $[4]$ & $\dfrac1{(2g+1)(2g-1)}$ \\ \hline
 $[3, 2^k, 3]$ with $k \in \ZZ_{\geq 0}$ & $\dfrac{k+2}{(2g+1)(4gk+6g-1)}$  \\ \hline
 $[2,3,2]$   &$\dfrac{1}{4g(2g+1)}$       \\ \hline
 $[2; [2], [2], [2^k, 3]]$ with $k \in \ZZ_{\geq 0}$& $\dfrac{1}{4(2kg+3g+k+1)(2g+1)}$ \\ \hline
 $x_0=[2, 3, 2^k]$ with $2 \leq k \leq 4$ &$\dfrac{2(gk-g-k-3)^2}{(2g+1)(3k+5)(8kg+8g+k-1)}$ \\ \hline
 $[2,4]$ & $\dfrac{2}{5 \cdot 7 \cdot 13}$ \\ \hline
\end{tabular}
\end{table}

\noindent\underline{\textbf{Case $x_0=[2,3,2^k]$ with $2 \leq k \leq 4$.}}
Suppose that $x_0=[2,3,2^k]$ with $2 \leq k \leq 4$.
Then $g \leq 4$ since $\frac6{11} \leq \frac{2k+2}{3k+5}=e_0 <a_1=\frac{g+1}{2g+1}$.
Assume in addition that $g \geq 3$.
Then $k=2$ since $\frac{2k+2}{3k+5}=e_0 <a_1=\frac{g+1}{2g+1} = \frac47$.
Hence $K_{S_0}^2=\frac{2^3}{7 \cdot 11 \cdot 73}$ or $\frac{2}{3^2 \cdot 11 \cdot 97}$, but by Lemma \ref{lem:gap}, $S_0$ cannot have singularities to compensate for the factor $73$ or $97$ in the denominator, a contradiction.
Thus $g=2$ and $K_{S_0}^2=\frac{2(k-5)^2}{5(3k+5)(17k+15)}=\frac{2 \cdot 3^2}{5 \cdot 11 \cdot 7^2}$, $\frac{2}{ 3 \cdot 5 \cdot 7 \cdot 11}$, and $\frac{2}{5 \cdot 17 \cdot 83}$ when $k=2$, $3$, and $4$ respectively.
Similarly, we conclude that $k=3$. 
Since $K_{S_0}^2+\Gap(x_0)=3+\frac{17 \cdot 31}{3 \cdot 5 \cdot 7 \cdot 11}$,
$S_0$ must have singularities $y_5$, $y_7$, and $y_{11} \neq x_0$ to compensate for the factors $5$, $7$, and $11$ in the denominator, respectively.
Lemma \ref{lem:totalgap} gives $\Gap(y_i) \leq 9$ and hence $y_5$, $y_7$, $y_{11}$, and $x_0$ are distinct from each other.
However, we obtain 
\begin{align*}
   &K_{S_0}^2+\Gap(y_5)+\Gap(y_7)+\Gap(y_{11})+\Gap(x_0) \\
   > &\Gap([3,2]) + \Gap([2,4]) + 2 \cdot \Gap([2,3,2^2])\\
   =&\frac{8}{5}+ \frac{6}{7}+2 \cdot \frac{38}{11}=9+\frac{141}{385}>9,
\end{align*}
a contradiction with Lemma \ref{lem:totalgap}.
Therefore $x_0 \neq [2,3,2^k]$ for $2 \leq k \leq 4$.

\noindent\underline{\textbf{Case $x_0=[4,2]$.}}
Suppose that $x_0=[4,2]$.
Then $g=2$ since $\frac47=e_0 < a_1=\frac{g+1}{2g+1}$.
Since $K_{S_0}^2=\frac{2}{5 \cdot 7 \cdot 13}$, 
$S_0$ must have singularities $y_5$ and $y_{13}$ to compensate for the factors $5$ and $13$ in the denominator, respectively.
By Lemma \ref{lem:totalgap}, we have $\Gap(y_5)+\Gap(y_{13}) < 9$.
Lemma \ref{lem:gap} now shows that $y_5$ and $y_{13}$ are $[3,2]$- and $[3,2^5]$-singularities, respectively.
On the other hand, by the shape of configuration III, there is a $[2]$-singularity on $S_0$, say $z$.
Then we have
\begin{align*}
K_{S_0}^2+\Gap(x_0)+\Gap(y_{5})+\Gap(y_{13})+\Gap(z)=9.
\end{align*}
By Lemma \ref{lem:totalgap}, we conclude that
$\Dyn(S_0)=[4,2]+[3,2]+[3,2^5]+m[4]+[2]$ for some $m \geq 0$. 
Moreover, $\Dyn(S_1)=[2]+[3, 2^5]+m[4]$ and only the $[2]$-singularity lies on $A_1$.
Now, as in the proof of \cite[Lemma 4.8]{Lac}, we obtain the following:
\begin{enumerate}
    \item $m=0$.
    \item The second hunt step extracts the $(-3)$-curve $E_2$ over the $[3,2^5]$-singularity.
    \item $(S_2, A_2+B_2)$ is a tacnode and $\pi_2$ is in configuration II of order five such that $\Sigma_2$ meets $E_2$ at the $[2^5]$-point.
    \item $\wt S_2$ is the Hirzebruch surface $\FF_2$ and $B_{2, \wt S_2}$ is a tautological section.
\end{enumerate}

Let $F \subset S_2$ be the ruling passing through the cusp of $A_2$ and let $h \colon S' \to S_2$ be the blow-up at the cusp.
Then, by contracting the strict transform of $F$, we obtain a birational morphism $S' \to \PP^2_k$.
Let $f \colon \wt S_0 \rightarrow \PP^2_k$ be the induced birational morphism,
and let $L \subset \wt S_0$ be the strict transform of the exceptional divisor of $h$.
Then $f(E_1)$ is a cuspidal cubic, $f(E_2)$ is a smooth conic intersecting with $f(E_1)$ at $f(\Sigma_2)$ with multiplicity five and at $f(F)$ transversally,
and $f(L)$ is a line passing through $f(F)$ and $f(\Sigma_1)$ such that $f(L)$ is tangent to $f(E_2)$ at $f(F)$.
By Lemma \ref{lem:specialconfig}, we conclude that $\Char(k)=5$ and $S_0$ is constructed as in Example \ref{eg:Lac4.8cex-1} with $(n_t, n_u)=(6,3)$.

In what follows, we may assume that $e(x_0) \leq \frac12$.

\noindent\underline{\textbf{Case $x_0=[2,3,2]$ or non-chain.}}
If $x_0$ is non-chain or $x_0 = [2,3,2]$, then $K_{S_0}^2=\frac1{4m}$ for some $m \in \Z$. By equation (\ref{eq:dioph3}), $S_0$ must have a singularity
to compensate for the factor 4 in the denominator, a contradiction with Lemma \ref{lem:gap}.
Hence $x_0$ is a chain singularity and $x_0 \neq [2,3,2]$.

\noindent\underline{\textbf{Case $x_0=[4]$.}}
Suppose that $x_0=[4]$.
Then equation (\ref{eq:dioph3}) gives us
\begin{align}\label{eq:diophantine3}
    \dfrac{1}{(2g+1)(2g-1)}+v_0+v_1+\sum_{k=0}^\infty n_k \cdot \dfrac{2(k+1)^2}{2k+3}=9,
\end{align}
where $v_0$ is the number of exceptional curves coming from Du Val singularities, $v_1$ is the sum of gaps coming from singularities on $S_0$ whose coefficients are $\frac12$, and $n_k$ is the number of points of type $[3,2^k]$. 
By \cite[Proposition 10.1]{KM} and Lemma \ref{lem:gap}, it holds that $v_1 \in \frac12\Z$. 
On the other hand, (\ref{eq:diophantine3}) shows that we can write $v_1$ as a fraction whose denominator is odd.
Hence $v_1 \in \Z$.
By the shape of configuration III, $S_0$ contains a $[2]$-singularity.
Hence $v_0 \geq 1$ and $v \coloneqq v_0+v_1 \geq 1$. 
From (\ref{eq:diophantine3}), we obtain that 
$v < 9$, $3v+2n_0 < 27$, and $v+\sum_{k=0}^\infty n_k \cdot k <9$. 
In particular, $n_k = 0$ for $k \geq 8$ and $\sum_{k=4}^7 n_k \leq 1$.
The latter shows that $4g^2-1$, the denominator of $K_{S_0}^2$, is a factor of $3^2 \cdot 5 \cdot 7 \cdot m$ with $m=11$, $13$, or $17$.
In particular, $4g^2-1 \leq 3^3 \cdot 5 \cdot 7 \cdot 17$ and hence $g \leq 63$.
Hence there are only finitely many possibilities for the tuple $(g, v, n_0, \ldots, n_7)$.
A computer search shows that (\ref{eq:diophantine3}) requires that $g=2$, $n_k=0$ for $k\geq 2$, and $(v,n_0,n_1) = (2m, 11-3m, 1)$ with $1 \leq m \leq 3$.

Since $a_1=\frac{g+1}{2g+1}=\frac35>\frac12$, we can take $M+F \in |K_{S_1}+A_1|$ as in \cite[\S 12]{KM}.
By the shape of configuration III, $A_1$ is contained in the smooth locus of $S_1$ and $S_1$ contains $(11-3m)[3]$-singularities away from $A_1$.
By \cite[Lemma 12.1]{KM}, it holds that $M =\emptyset$, $K_{S_1}+F$ is negative, and $F$ contains the $(11-3m)[3]$-singularities.
Hence $(v,n_0)=(6,2)$.
If there is a non-Du Val singularity, say $y$, on $S_1$ distinct from the $2[3]$-singularities, then [\textit{ibid}.] shows that $y \in F$ and
$(K_{S_1}+F \cdot F) \geq -2+3 \cdot \frac23=0$, a contradiction.
Hence $\Sing(S_1)$ consists of $2[3]$-singularities and Du Val singularities of total index $v-1=5$.
Note that $A_1^2=-4+10=6$.

Now $(S_1, A_1)$ is the same as that of Case 1 of the proof of Lemma \ref{Lac4.7} (or \cite[Lemma 4.7]{Lac}).
As a result, we can conclude that $\Char(k)=2$ and there is a birational morphism $f \colon \wt S_0 \to \PP^2_k$ such that $f(E_2)$ is a cuspidal cubic, 
$f(A_1)$ is a conic intersecting with $f(E_2)$ at one point with multiplicity six,
and $f(L)$ is a line intersecting with $f(E_2)$ only at its cusp.
However, as we pointed out in the proof of Lemma \ref{Lac4.7}, $f(L)$ must be tangent to $f(A_1)$ since $A_1$ is cuspidal, a contradiction with Lemma \ref{lem:specialconfig} (5).
Therefore $x_0 \neq [4]$.

\noindent\underline{\textbf{Case $x_0=[3,2^k,3]$.}}
Suppose that $x_0=[3,2^k,3]$ for some $k \geq 0$.
Then equation (\ref{eq:dioph3}) gives us
\begin{align}\label{eq:diophantine4}
    \dfrac{k+2}{(2g+1)(4gk+6g-1)}+v_0+v_1+\sum_{i=0}^\infty n_i \cdot \dfrac{2(i+1)^2}{2i+3}=9,
\end{align}
where $v_0$ is the number of exceptional curves coming from Du Val singularities, $v_1$ is the sum of gaps coming from singularities on $S_0$ whose coefficients are $\frac12$, and $n_i$ is the number of points of type $[3,2^i]$. 
By \cite[Proposition 10.1]{KM} and Lemma \ref{lem:gap}, it holds that $v_1 \in \frac12\Z$. 
On the other hand, (\ref{eq:diophantine4}) shows that we can write $v_1$ as a fraction whose denominator is odd.
Hence $v_1 \in \Z$.

By Lemma \ref{lem:gap}, we have $k+1= \Gap(x_0)\leq v_1 < 9-v_0$.
By the shape of configuration III, there is a $[2]$-singularity on $\overline{\Sigma}_1$, and hence $v_0 \geq 1$.
In particular, $k \leq 7$.
From (\ref{eq:diophantine4}), we obtain $\sum_{i=0}^\infty n_i \cdot i \leq \sum_{i=0}^\infty n_i \cdot \frac{2(i+1)^2}{2i+3} <9-v_1 \leq 8-k$.
In particular, $n_i=0$ for $i \geq 8-k$.
Hence $K_{S_0}^2=\frac{k+2}{(2g+1)(4gk+6g-1)}$ must be expressed as a fraction whose denominator is composed exclusively of factors of $2i+3$ with $i \leq 7-k$.

Assume that $k=7$.
Then $K_{S_0}^2=\frac{9}{(2g+1)(34g-1)} \in \frac{1}{3} \Z$.
Hence $(2g+1)(34g-1) \leq 27$, a contradiction with $g \geq 2$.

Assume that $k=6$.
Then $K_{S_0}^2=\frac{8}{(2g+1)(30g-1)} \in \frac1{3 \cdot 5} \ZZ$, which implies that $(2g+1)(30g-1) \leq 15$, a contradiction.

Assume that $k=5$.
Then $K_{S_0}^2=\frac{7}{(2g+1)(26g-1)} \in \frac1{3 \cdot 5 \cdot 7} \ZZ$, which implies that $(2g+1)(26g-1) \leq 3 \cdot 5 \cdot 7^2$.
Hence $2 \leq g \leq 3$.
However, $(2g+1)(26g-1)$ has a factor $17$ or $11$, a contradiction.

Assume that $k=4$.
Then $K_{S_0}^2=\frac{6}{(2g+1)(22g-1)} \in \frac1{3^2 \cdot 5 \cdot 7} \ZZ$, which implies that $(2g+1)(22g-1) \leq 3^3 \cdot 5 \cdot 7$.
Hence $2 \leq g \leq 4$.
However, $(2g+1)(22g-1)$ has a factor $43$, $13$, or $29$, a contradiction.

Assume that $k=3$. Then 
\begin{align}\label{eq:genus1}
   K_{S_0}^2=\frac{5}{(2g+1)(18g-1)} \in \frac1{3^2 \cdot 5 \cdot 7 \cdot 11} \ZZ, 
\end{align}
which implies that $(2g+1)(18g-1) \leq 3^2 \cdot 5^2 \cdot 7 \cdot 11$.
Hence $2 \leq g \leq 21$.
A computer search shows that (\ref{eq:genus1}) requires that $g=2$. 
Since $K_{S_0}^2 = \frac1{5 \cdot 7}$, there are $[3,2]+[3,2^2]$-singularities on $S_0$.
However, we cannot write $9-\Gap([3, 2^3,3])-\Gap([3, 2])-\Gap([3,2^2])-K_{S_0}^2 =\frac{29}{35}$ as the sum of gaps of klt singularities with coefficients $\leq \frac12$, a contradiction.

Assume that $k=2$.
Then $\sum_{i=0}^5 n_i \cdot i \leq \sum_{i=0}^\infty n_i \cdot \frac{2(i+1)^2}{2i+3} <9-v_1 \leq 8-k=6$.
In particular, there is at most one singularity on $S_0$ of the form $[3,2^l]$ with $3 \leq l \leq 5$. 
Then $K_{S_0}^2=\frac{4}{(2g+1)(14g-1)} \in \frac1{3 \cdot 5 \cdot 7 \cdot a} \ZZ$ with $a=3, 11$, or $13$, which implies that $(2g+1)(14g-1) \leq 3 \cdot 5 \cdot 7 \cdot 13$.
Hence $2 \leq g \leq 6$.
However, $(2g+1)(14g-1)$ has a factor $3^3$, $41$, $3^2 \cdot 11$, $23$, or $83$, a contradiction.

Assume that $k=1$.
Then, similarly, there is at most one singularity on $S_0$ of the form $[3,2^l]$ with $3 \leq l \leq 6$. 
Since $e([3,2^6])=\frac{2 \cdot 7^2}{3 \cdot 5}$, we have
\begin{align}\label{eq:genus2}
K_{S_0}^2=\frac{3}{(2g+1)(10g-1)} \in \frac1{3 \cdot 5 \cdot 7 \cdot a} \ZZ \text{ with $a=3, 11$, or $13$}, 
\end{align}
which implies that $(2g+1)(10g-1) \leq 3^2 \cdot 5 \cdot 7 \cdot 13$.
Hence $2 \leq g \leq 14$.
A computer search shows that (\ref{eq:genus2}) has no solution, a contradiction.

Therefore $k=0$.
Then $\sum_{i=0}^7 n_i \cdot i \leq \sum_{i=0}^\infty n_i \cdot \frac{2(i+1)^2}{2i+3} <9-v_1 \leq 8-k=8$.
In particular, there is at most one singularity on $S_0$ of the form $[3,2^l]$ with $4 \leq l \leq 7$.
Then 
\begin{align}\label{eq:genus3}
K_{S_0}^2=\frac{2}{(2g+1)(6g-1)} \in \frac1{3^2 \cdot 5 \cdot 7 \cdot a} \ZZ \text{ with $a=11, 13$, or $17$}, 
\end{align}
which implies that $(2g+1)(6g-1) \leq 3^2 \cdot 5 \cdot 7 \cdot 17$.
Hence $2 \leq g \leq 20$.
A computer search shows that (\ref{eq:genus3}) requires that $(g, a)=(2,11)$, $(3,17)$, or $(6,13)$. 

Assume that $g=6$ in addition.
Then $K_{S_0}^2 = \frac{2}{5 \cdot 7 \cdot 13}$.
Hence $n_2 \geq 1$ and $n_6 \geq 1$, a contradiction with $\sum_{i=0}^7 n_i \cdot i <8$.

Assume that $g=3$ in addition.
Then $K_{S_0}^2 = \frac{2}{7 \cdot 17}$.
Hence $n_2 \geq 1$ and $n_8 \geq 1$, a contradiction with $\sum_{i=0}^7 n_i \cdot i <8$.

Thus $g=2$.
Since $K_{S_0}^2= \frac{2}{5 \cdot 11}$, we have $n_4 \geq 1$ and either $n_1 \geq 1$ or $n_6 \geq 1$.
In the latter case, we obtain $\sum_{i=0}^7 n_i \cdot i \geq 8$, a contradiction.
Hence $S_0$ contains the $[3^2]+[3,2]+[3,2^4]$-singularities. 
Since $9-\Gap([3, 2^4])-\Gap([3, 2])-\Gap([3^2])-K_{S_0}^2=\frac{102}{5 \cdot 11}$, $S_0$ must contain another $[3, 2^4]$-singularity, but $\Gap([3, 2^4]) > \frac{102}{5 \cdot 11}$, a contradiction.

Therefore $x_0 \neq [3,2^k, 3]$ for any $k \geq 0$.
In what follows, we may assume that $e_0<\frac12$.

\noindent\underline{\textbf{Case $x_0=[3,2^r]$.}}
Suppose that $x_0=[3,2^r]$. Then equation (\ref{eq:dioph3}) gives us

\begin{align}\label{eq:diophantine5}
    \dfrac{2(r+g+2)^2}{(2r+3)(2g+1)(4gr+4g-1)}+v+\sum_{k=0}^r n_k \cdot \dfrac{2(k+1)^2}{2k+3}=9,
\end{align}
where $v$ is the number of exceptional curves coming from Du Val singularities, and $n_k$ is the number of points of type $[3,2^k]$. 
From this formula, we obtain that $v < 9$, $3v+2n_0 < 27$, and $v+\sum_{k=0}^r n_k \cdot k <9$. 
In particular, we have $r \leq 8$.
By the maximality of $r$, we have $n_k = 0$ for $k > r$.
By the shape of configuration III, there are $[3,2^{g-1}]$- and $[2]$-singularities on $\overline{\Sigma}_1$. 
In particular, we have $2 \leq g \leq r+1$ and $v \geq 1$.
Hence there are only finitely many possibilities for the tuple $(g, v, r, n_0, \ldots, n_8)$.
A computer search shows that (\ref{eq:diophantine5}) requires that $g=2$, $r=1$, and $(v,n_0,n_1) = (1+2m, 7-3m,2)$ with $0 \leq m \leq 2$.

In particular, $\Sing(S_1)$ consists of a singularity $w$ on $A_1$, which is a $[2]$-singularity, $n_0[3]$-singularities, and Du Val singularities of total index $(v-1)$ disjoint from $A_1$.
  
Since $a_1=\frac{g+1}{2g+1}=\frac35>\frac12$, we can take $M+F \in |K_{S_1}+A_1|$ as in \cite[\S 12]{KM}.
By the shape of configuration III, $A_1$ is contained in the smooth locus of $S_1$ and $S_1$ contains $(7-3m)[3]$-singularities away from $A_1$.
By \cite[Lemma 12.1]{KM}, it holds that $M$ is either empty or irreducible, $K_{S_1}+M+F$ is negative, and $M+F$ contains the $(7-3m)[3]$-singularities.
The negativity shows that both $F$ and $M$ cannot contain $3[3]$-singularities.
On the other hand, $F \cap M$ is a single point by [\textit{ibid}.].
Hence $(v,n_0)=(5,1)$.
The next hunt step extracts the $(-3)$-curve since $e(w, K_{S_1}+a_1A_1)=\frac{3}{10}<\frac13 =e([3])$. 

Suppose that $T_2$ is a net.
Then $\Dyn(T_2)=2[2]+[2^3]$ by \cite[11.5.9 Lemma]{KM}.
Let $F_1$ and $F_2$ be the fibers containing $2[2]$ and $[2^3]$, respectively.
On the other hand, $(A_2 \cdot \Sigma_2) \geq 2$ and $(E_2 \cdot \Sigma_2) \geq 1$.
If $(E_2 \cdot \Sigma_2) \geq 3$, then $(K_{T_2}+a_1A_2+e_1E_2 \cdot \Sigma_2) \geq -2+\frac65+1>0$, a contradiction.
If $E_2$ is a section, then $T_2$ is smooth, a contradiction.
Hence $E_2$ is a $2$-section contained in $T_2^\circ$.
Then $(E_2 \cdot F_1)=(E_2 \cdot F_2)=1$ since $F_1$ and $F_2$ are of multiplicity two.
Taking the minimal resolution of $T_2$ and running MMP over $\PP^1_k$, we obtain a birational map $T_2 \dashrightarrow H$ to a Hirzebruch surface such that the strict transform of $E_2$ in $H$ is a smooth $2$-section whose self intersection number is two, a contradiction with \cite[11.5.10 Lemma]{KM}.
Therefore $T_2$ is not a net.

Suppose that $\Sigma_2$ intersects with $A_2$ at the $[2]$-singularity in $T_2$.
Then $K_{T_2}+\Sigma_2$ is plt at the $[2]$-singularity by contractibility.
In particular, $K_{S_2}^2=5$ and $(A_2 \cdot \Sigma_2)=m+\frac12$ for some $m \in \Z_{\geq 0}$.
Since $0>(K_{T_2} + a_1 A_2 +e_1 E_2 \cdot \Sigma_2) \geq -1+\frac{6m+3}{10}+\frac13$, we have $m=0$.
Thus $\pi_2$ induces an isomorphism on $A_1$.
In particular, $(K_{S_2}+A_2, A_2)=2$.
By the shape of configuration III, we have $A_2^2=8$, which leads to $(K_{S_2} \cdot A_2)=-2$.
However, this implies that $(K_{S_2} \cdot A_2)^2 \neq K_{S_2}^2 \cdot A_2^2$, a contradiction.

Therefore $\Sigma_2 \cap A_2 \subset T_2^\circ$. 
Since $0>(K_{T_2}+a_1A_2+e_1E_2 \cdot \Sigma_2)=-1+\frac35(A_2 \cdot \Sigma_2)+\frac13(E_2 \cdot \Sigma_2)$, we have $(E_2 \cdot \Sigma_2)=(A_2 \cdot \Sigma_2)=1$ in $T_2$.
In particular, $E_2 \subset S_2^\circ$ is a smooth rational curve.
Since a $[2]$-singularity lies on $A_2 \subset S_2$, \cite[13.7 Lemma]{KM} now shows that $\wt S_2$ is the Hirzebruch surface $\FF_2$.
Hence $\Dyn(S_1)=[3]+[2]+[2^4]$ and $\Sigma_2$ passes through the $[2^4]$-singularity.
In particular, $(S_2, A_2+B_2)$ is a tacnode and $\pi_2$ is in configuration I of order five and $B_2$ is a tautological section.
Therefore we obtain the same configuration $(S_2, A_2+B_2)$ as in the case where $x_0=[4,2]$.
Furthermore, a similar analysis shows that $\Char(k)=5$ and $S_0$ is constructed as in Example \ref{eg:Lac4.8cex-1} with $(n_t, n_u)=(5,3)$. 
\end{proof}

To treat the case where $\pi_1$ is in configuration $u$ or $v$, we make \cite[Lemma 15.3]{KM} characteristic-free.

\begin{lem}[{\cite[Lemma 15.3]{KM}}]\label{lem:UVF}
Suppose that $A_1$ has a double point, $g(A_1) \geq 2$, and that the configuration is $u$ or $v$.
Then $S_1$ has at least two singular points along $F$ as in \cite[\S 12]{KM}.
\end{lem}

\begin{proof}
Take $F+M \in |K_{S_1}+A_1|$ as in \cite[\S 12]{KM}.
Then $F \not \subset S_1^\circ$.
Assume that $F \cap \Sing(S_1)$ is a single point, say $b$.
Let $f \colon Z \to \PP^1_k$ be the $\PP^1_k$-fibration as in \cite[12.3]{KM} obtained by extracting the unique divisor $E$ adjacent to $F$.
Let $\Gamma$ be the $f$-fiber passing through the cusp of $A_{1, Z}$. 
Then, as in the proof of \cite[Lemma 15.3]{KM}, we obtain the following.
\begin{enumerate}
    \item $E$ is a section of $f$.
    \item $f$ has at least one multiple fiber.
    \item $b \not \in A$.
    \item $K_{S_1}+\Gamma_{S_1}$ is not lc at $b$.
\end{enumerate}
Since the first hunt step is an isomorphism at $b \in S_1$, we have $e(b) \leq e_0 <a_1<\frac23$.
By \cite[12.3 (3)]{KM}, it holds that $E_{\wt S_1}^2 \leq -3$.

Let $\{F_1, \ldots, F_n\}$ be the set of multiple fibers of $f \colon Z \to \PP^1_k$.
Write $v_i=F_i \cap E$, where $v_i \in Z$ must be singular.
Then $n \leq 3$ since otherwise $b \in S_1$ is not klt.
If $n=1$, then $K_{S_1}+\Gamma_{S_1}$ is plt at $b$, a contradiction.

Assume that $n=3$.
Then $b$ must be a non-chain singularity whose central curve is $E_{\wt S_1}$.
Since $E_{\wt S_1}^2 \leq -3$, \cite[Proposition 10.1 (2)(d)]{KM} shows that $e(b) \geq \frac23$, a contradiction.

Hence $n=2$. 
By a similar argument, both $v_1$ and $v_2$ must be chain singularities.
Let $\wt f \colon \wt S_1 \to Z \to \PP^1_k$ be the induced morphism.

Let us classify the possibilities for $F_i \cap \Sing(Z)$.
Suppose that $F_i \cap A_{1,Z}$ contains a smooth point of $Z$.
Then $F_i$ must be of multiplicity two and $F_i \cap A_{1,Z}$ is a single point.
As in \cite[Lemma 11.5.13]{KM}, we can show that $F_i \cap \Sing(Z)$ is a $2[2]$-, $[2^3]$-, or $[2;[2],[2],[2^{k-3}]]$-singularity for some $k \geq 4$.
The last case cannot occur since $v_i$ is a chain singularity.

On the other hand, suppose that $F_i \cap A_{1,Z}$ contains a singular point of $Z$.
By \cite[11.5.5]{KM}, $F_i$ can contain at most two singular points.
Since $v_i \in F_i$, it holds that $F_i \cap A_{1,Z}$ is exactly one singular point of $Z$, say $w_i$.
Since $a_1<\frac23$, by \cite[Lemma 8.0.7 (2)]{KM}, $w_i \in A_{1,Z}$ is of spectral value at most one, and so by \cite[Lemma 8.0.8]{KM}, is an (almost) Du Val singularity such that the exceptional divisor adjacent to $A_{1,Z}$ is a component of $\wt f^{-1}(F_i)$ of multiplicity two.
\cite[Lemma 11.5.9]{KM} now shows that $F_i$ is of type (4) or (5) with $k=2$.
In the former (resp.\ latter) case, $v_i$ is a $[2,3,2]$- (resp.\ $[3]$-) singularity and $w_i$ is a $[2]$- (resp.\ $[2^2]$-) singularity.

Take $\wt S_1 \to H$ to be the MMP/$\PP^1_k$ isomorphic along $E_{\wt S_1}$.
Then $E_H$ is the negative section of $H$ disjoint from $A_{1,H}$.
By the shape of the $F_i$'s, we can check that the induced morphism $A_1 \to A_{1,H}$ is an isomorphism.
Since $g(A_1)=2$, it holds that $H$ is the Hirzebruch surface $\FF_3$, $E^2=E_H^2=-3$, and $A_H^2=12$.

Now we check the value of $e(b)$ for each case.
If one of the $v_i$'s is a $[2,3,2]$-singularity, then $e(b) \geq e([2,3,2,3,2]) = \frac23>a_1$, a contradiction.

Suppose that one of the $v_i$'s is a $[3]$-singularity for some $i$.
If the other is a $[3]$- (resp.\ $[2^3]$-) singularity, then $e(b)=e([3,3,3])=\frac57$ (resp.\ $e([3,3,2^3])=\frac{16}{23}$) $>\frac23$, a contradiction.
Thus the other is a $[2]$-singularity and $e(b)=e([3,3,2])=\frac{8}{13}$.
By the shape of the $F_i$'s, we have $A_{\wt S_1}^2=8$.
Hence $E_1^2=-2$ (resp.\ $-3$) and $x_0$ is a $[2^2, 2, 2^2]$- (resp.\ $[3,2^2]$)-singularity when $\pi_1$ is in configuration $u$ (resp.\ $v$).
However, we have $e(b)>e_0$, a contradiction with the maximality of $e_0$.

Suppose that one of the $v_i$'s is a $[2^3]$-singularity for some $i$.
If the other is also a $[2^3]$-singularity, then $e(b)=e([2^3,3,2^3])=\frac23$, a contradiction.
Thus the other is a $[2]$-singularity and $e(b)=e([2,3,2^3])=\frac{4}{7}$.
By the shape of the $F_i$'s, we have $A_{\wt S_1}^2=7$.
Hence $E_1^2=-3$ (resp.\ $-4$) and $x_0$ is a $[3, 2^2]$- (resp.\ $[4]$)-singularity when $\pi_1$ is in configuration $u$ (resp.\ $v$).
However, we have $e(b)>e_0$, a contradiction with the maximality of $e_0$.

Suppose that both $v_1$ and $v_2$ are $[2]$-singularities.
Then $b$ is a $[2,3,2]$-singularity and $K_{S_1}+\Gamma_{S_1}$ is lc at $b$, a contradiction.
Combining these results, we obtain the assertion.
\end{proof}

\begin{lem}[{\cite[Lemma 4.9]{Lac}}]
Suppose that $A_1$ has a double point, $g(A_1) \geq 2$, and that the configuration $u$ or $v$ arises.
Then $\Char(k)=5$, the configuration $v$ arises, and $S_0$ is constructed as in Example \ref{eg:Lac4.8cex-1} with $(n_t, n_u)=(5,4)$.
\end{lem}

\begin{proof}
By Lemma \ref{lem:UVF}, there are at least two singular points along $F$.
Then the proof goes exactly along the lines of \cite[Lemma 4.9]{Lac}, with one modification.
Indeed, if $A_1 \subset S_1^\circ$, then \cite[15.4.1]{KM} cannot exclude only the case where $F$ must pass through exactly the $2[3]$-singularities.

To exclude this case, let us assume that $A_1 \subset S_1^\circ$ and that $F$ passes through exactly the $2[3]$-singularities.
By \cite[Lemma 12.3]{KM}, $\Sing(S_1)$ consists of $2[3]$-singularities and Du Val singularities.
Set $m=-(E_{1,\wt S_0})^2$ and let $v$ be the total index of Du Val singularities.

Assume in addition that $\pi_1$ is in configuration $u$.
Then $\Sing(S_0)$ consists of an $[m,2^2]$-singularity, an $[4,2]$-singularity, $2[3]$-singularities, and Du Val singularities.
Furthermore, $x_0$ is an $[m,2^2]$-singularity.
Since 
\begin{align*}
e([4,2]) =\frac47 \leq e_0 = \frac{3m-6}{3m-2} <a_1=\frac{9}{14},
\end{align*}
we obtain $m=4$ and $A_1^2=-4+10=6$.
Since the Picard rank of $S_0$ is one, we have 
\begin{align*}
    K_{S_0}^2=\dfrac{(K_{S_0} \cdot \overline{\Sigma}_1)^2}{\overline{\Sigma}_1^2}
    =\dfrac{(\frac{1}{35})^2}{\frac{3}{35}}=\dfrac{1}{105}.
\end{align*}
Then equation (\ref{eq:dioph3}) gives us
\begin{align*}
    v
    &=9-\left(K_{S_0}^2+2\Gap([3])+\Gap([4,2,2])+\Gap([4,2])\right)\\
    &=9-\left(\dfrac{1}{105}+2 \cdot \frac23 + \frac95 + \frac67\right)=5.
\end{align*}

On the other hand, assume in addition that $\pi_1$ is in configuration $v$.
Then $\Sing(S_0)$ consists of an $[m]$-singularity, a $[2,3,2^2]$-singularity, $2[3]$-singularities, and Du Val singularities.
Furthermore, $x_0$ is an $[m]$-singularity.
Since 
\begin{align*}
e([2,3,2^2])=\frac6{11} \leq e_0=\frac{m-2}{m} <a_1=\frac{7}{11},  
\end{align*}
we obtain $m=5$ and $A_1^2=-4+10=6$.
Since the Picard rank of $S_0$ is one, we have 
\begin{align*}
    K_{S_0}^2=\dfrac{(K_{S_0} \cdot \overline{\Sigma}_1)^2}{\overline{\Sigma}_1^2}
    =\dfrac{(\frac{2}{55})^2}{\frac{6}{55}}=\dfrac{2}{165}.
\end{align*}
Then equation (\ref{eq:dioph3}) gives us
\begin{align*}
    v
    &=9-\left(K_{S_0}^2+2\Gap([3])+\Gap([5])+\Gap([2,3,2^2])\right)\\
    &=9-\left(\dfrac{2}{165}+2 \cdot \frac23 - \frac45 + \frac{38}{11}\right)=5.
\end{align*}

Therefore, for each case, we have $A_1^2=6$ and $\Sing(S_1)$ consists of $2[3]$-singularities and Du Val singularities of total index five.
Now $(S_1, A_1)$ is the same as that of Case 1 of the proof of Lemma \ref{Lac4.7} (or \cite[Lemma 4.7]{Lac}).
As in the case $x_0=[4]$ of the proof of Lemma \ref{Lac4.8}, we obtain a contradiction.
\end{proof}

\begin{proof}[Proof of Proposition \ref{Lac4.5}]
This follows from \cite[Lemmas 11.1.1 and 11.2.1]{KM}, and the lemmas in this subsection.
\end{proof}

\subsection{$A_1$ has a simple cusp}
\cite[\S 4.3]{Lac} is devoted to the case where $A_1$ has a simple cusp.
In $\Char(k)=0$, such a log del Pezzo surface $S_0$ is classified in \cite[\S 16]{KM} under the additional assumption that $\pi_1^{\mathrm{alg}}(S_0^\circ) =0$.
This assumption and the Bogomolov bound are used in the proof of \cite[Lemma 16.5]{KM}, but we can rewrite the proof to avoid using them.
Indeed, the Bogomolov bound is used only to show that there is no singularity on $S_2$ other than $y$ and a $[2]$-singularity in the case where $A_2 \not\subset S_2^\circ$, but we do not use this fact later.
The assumption $\pi_1^{\mathrm{alg}}(S_0^\circ) =0$ is used only to check that $S_2=S(A_1)$, $S(A_1+A_2)$, or $S(A_4)$ in the fourth paragraph of the proof, but in $\Char(k) \neq 2,3$, we can also check it by \cite[Lemma B.9]{Lac}.
Note that the assumption $S \neq \PP^2_k$, $S(A_1)$ is omitted in [\textit{ibid}.].
In the remainder of the proofs in \cite[\S 16]{KM}, the assumption that $\pi_1^{\mathrm{alg}}(S_0^\circ) =0$ and the Bogomolov bound are only used in that of \cite[Lemma 16.3 (1)-(5)]{KM}.
For this reason, Lacini gives counterparts of [\textit{ibid}.], which are \cite[Lemmas 4.11-4.15]{Lac}.
The assumption $\Char(k) \neq 2,3$ is used only in the following.
In several arguments, we use \cite[Lemma B.9]{Lac} to $S_1$, $S_2$, or $W$ as in \cite[12.4 Lemma-Definition]{KM} with $S=S_1$ or $S_2$.
In the proof of \cite[Lemma 4.14]{Lac}, we use \cite[Theorem B.6]{Lac} to determine $\Dyn(S_1)$.
We also use \cite[Lemma 11.5.11]{KM} in \cite[Lemma 4.15]{Lac} to show that $T_2$ is not a net.

\begin{rem}\label{rem:Section4.3}
Let us mention some remarks on the proofs in \cite[\S 4.3]{Lac}.
    \begin{enumerate}
        \item In Case 2 of the proof of \cite[Lemma 4.11]{Lac}, we conclude that, if $a_2 \geq \frac56$, then there is a tiger of $(S_2, a_2A_2+b_2B_2)$ over the cusp of $A_2$. We obtain this tiger by blowing-up at the cusp three times along $A_2$.
        In particular, this tiger is not contained in $\tilde{S}_0$ when $\pi_1$ is in configuration I or II.
        Since we only assume that $S_0$ has no tigers in $\wt S_0$, this does not yet lead to a contradiction.
        However, we have already treated the case where $\pi_1$ is in configuration I in \cite[Lemma 4.6]{Lac} (or Lemma \ref{Lac4.6}).
        Furthermore, we can conclude that $\pi_1$ is not in configuration II as follows.
        
        Recall that $S_2$ is the Du Val del Pezzo surface of type $[2]+[2^2]$, $A_2$ is its cuspidal anti-canonical member, and $B_2$ is the $(-1)$-curve passing through the $[2]$- and $[2^2]$-singularities.

        Suppose in addition that $x_1 \in A_1$.
        Then $\Sigma_2 \cap A_{1, T_2} = \emptyset$ since $(S_2, A_2+B_2)$ is a fence.
        By the shapes of $\pi_2$, we conclude that $A_{1, \wt S_1}^2=A_2^2=6$.

        Suppose in addition that $x_1 \not \in A_1$.
        If $E_2 \cap \Sigma_2 \in T_2$ is a singular point, then $x_1$ is a non-chain singularity such that two of the three branches are $[2]$ and $[2^2]$.
        However, this implies that $x_1$ is Du Val or of coefficient $\geq \frac23=a_1$, a contradiction.
        Hence $E_2 \cap \Sigma_2$ is a smooth point and $\pi_2$ is in configuration I or (II, $x^{r-1}$) for some $r \geq 1$ by \cite[Lemma 11.1.1]{KM}.
        By the shapes of $\pi_2$, we conclude that $A_{1, \wt S_1}^2=A_2^2-1=5$.

        Thus, in each case, we have $A_{1, \wt S_1}^2 \geq 5$.
        If $\pi_1$ is in configuration II, then $E_1^2 \geq 5-5=0$, a contradiction.

        Consequently, the statement of \cite[Lemma 4.11]{Lac} still holds.

        \item In the situation of Cases 2a and 2b of the proof of \cite[Lemma 4.12]{Lac}, the equalities $K_W^2=(K_W \cdot G_W)=1$ do not imply that the pullback of $G_W$ is in $|-K_{\wt{W}}|$. 
        The following is a modification of the arguments of the original proof.

        By \cite[Lemma 12.1 (1)]{KM}, (the reduced structure of) $M \subset S_1$ is isomorphic to $\PP^1_k$.
        Then $G$ should meet $M_Y$ at one point transversally.
        Hence $G_W \cong G \cong \mathbb{P}^1_k$. 
        
        Indeed, there are exactly four smooth rational curves in $S=S(2A_4)$ of anti-canonical degree one.
        Such a curve $L$ passes through $2[2^4]$-singularities at one of which $K_S+L$ is not plt (see \cite[Example B.11]{Lac} for more detail).
        
        Consequently, we have shown that $G_W$ is smooth. However, the case where $W=S(2A_4)$ must still be considered in the remainder of the proof of \cite[Lemma 4.12]{Lac}.

        \item In Case 2 of the proof of \cite[Lemma 4.12]{Lac}, it is written that we can proceed as in \cite[Lemma 4.11]{Lac} in the case where $K_W+G_W$ is dlt. 
        As we said in (1), we have to show in addition that $\pi_1$ is not in configuration II. 
        The following is the proof.
        
        We have $A_{1,\wt S_1}^2=6$ since the center of $\pi$ cannot be $A_{1,W} \cap G_W$ since $K_Y+G$ is plt.
        Hence $A_{1, \wt S_1}^2=6$.
        If $\pi_1$ is in configuration II, then $E_1^2 = 6-5=1$, a contradiction.

        Consequently, in Case 2 of the proof of \cite[Lemma 4.12]{Lac}, we may assume that $K_W+G_W$ is not dlt.

        \item In Case 2 of the proof of \cite[Lemma 4.12]{Lac}, we have to deny in addition that $W=S(2A_4)$ with $G_W$ smooth as we mentioned in (2).
        The following is the proof.
        
        Suppose that $W=S(2A_4)$ and $G_W$ is smooth.
        Then $G_W$ passes through the $2[2^4]$-singularities, at one of which $K_W+G_W$ is not plt.
        Since $K_Y+G$ is plt and $G$ passes through at most one singularity, we conclude that $G$ is a $-(2+r)$-curve and $M_Y$ passes through a $[3;[2],[2^2],[2^r]]$-singularity (or $[2,3,2^2]$-singularity when $r=0$).
        In particular, the image of $G$ in $S_1$ is a $[2+r, 2^4]$-singularity on $A_1$.
        Since its spectral value equals $5r$, \cite[Lemma 8.0.7 and 11.2.1 (13)]{KM} shows that $\frac{5r}{5r+1}<a_1<\frac45$.
        Hence $r=0$.
        When $\pi_1$ is in configuration II (resp.\ III, $u$, $(u, n)$, $v$, $(v, n)$, $(v,n^2)$, $(v, f)$, $(v, f^2)$, and $w$), we have $x_0=[3,2^5]$ (resp.\ $[2,4,2^5]$, $[5,2^5]$, $[2^2,5,2^5]$, $[2,3,5,2^5]$, $[6,2^5]$, $[7,2^5]$, $[8,2^5]$, $[2,6,2^5]$, $[2^2, 6, 2^5]$, and $[3,2,5,2^5]$).
        In each case, applying \cite[Lemma 11.2.1]{KM}, we can check that $e_0>a_1$, a contradiction.

        Consequently, the statement of \cite[Lemma 4.12]{Lac} still holds.

        \item In Case 2 of the proof of \cite[Lemma 4.14]{Lac}, we should show that configuration III cannot arise.
        The following is the proof:
        Suppose that configuration III arises.  
        Then there are exactly two non-Du Val singularities on $S_0$, which are $x_0$ and an $[3]$-singularity.
        In particular, \cite[(10.3)]{KM} gives 
        \begin{align*}
        h^0(-K_{\overline{S}}-D) &\geq 1-2+K_{S_0}^2+\Gamma\\
        & \geq  -1+K_{S_0}^2+(1-e([3]))+(1-e(x_0))\\
        & >  1-e([3])-a_1=0.
        \end{align*}
        In particular, any exceptional divisor over a non-Du Val point is a tiger, a contradiction.

        \item To finish the proof of \cite[Proposition 4.10]{Lac}, it remains to show the following: If $A_1$ has a simple cusp and $S_1$ is not Gorenstein, then $(S_2, A_2+B_2)$ is not a tacnode.
        An analysis similar to that in the proof of \cite[Lemma 18.5]{KM} (or \cite[Lemma 16.3 (6)]{KM}) shows this assertion.

        Consequently, the statement of \cite[Proposition 4.10]{Lac} still holds.
    \end{enumerate}
\end{rem}

\subsection{$A_1$ has a simple node}\label{LacS4.4}

\cite[\S 4.4]{Lac} is devoted to the case where $A_1$ has a simple node.
In $\Char(k)=0$, such a log del Pezzo surface $S_0$ is classified in \cite[\S 17]{KM} under the additional assumption that $\pi_1^{\mathrm{alg}}(S_0^\circ) =0$.
This assumption is used to obtain \cite[Proposition 17.3 (2) and (3)]{KM}.
This is why we are concentrated on proving \cite[Proposition 4.19 (2) and (3)]{Lac}.
The assumption $\Char(k) \neq 2,3$ is used only to ensure the uniqueness of the surface as stated in \cite[Lemma B.10]{Lac}, and to guarantee the correctness of the classification given by \cite[Proposition 13.5]{KM}.
We also use \cite[Lemma 11.5.11]{KM} in \cite[Lemma 16.6]{KM} to show that $T_2$ is not a net.

\begin{rem}
Let us mention some remarks on \cite[Proposition 4.19]{Lac}.
    \begin{enumerate}
        \item The inequality as in \cite[Proposition 17.3 (3)]{KM} should be replaced by $b<\frac67$ because we only know that $\frac{b}{2}+3a<3$ and $a>b$.
        Similarly, the inequality as in \cite[Proposition 4.19(1)(a)]{Lac} should be replaced by $b_2<\frac67$.
        
        \item In the setting of \cite[Proposition 4.19]{Lac}, we can deduce $a_2+b_2>1$ as follows.
        By \cite[Lemma 4.18 (6) and (7)]{Lac}, it holds that $a_1 \geq \frac23$ and there are exactly two non-Du Val points, say $x_0$ and $y$, on $S_0$.
        Since $A_1$ has a simple node, $\pi_1$ is as listed in \cite[Lemma 11.1.1]{KM}.
        In particular, $y$ is preserved via the first hunt step.
        By \cite[Proposition 10.1]{KM}, we obtain $a_2+b_2 > a_1 + e(y) \geq 1$.
        
        \item The following is the reason why we only have to consider the case where $(S_2, A_2+B_2)$ is a fence or a tacnode in the proof of \cite[Proposition 4.19]{Lac}.
        In such a setting, $A_2$ is not contracted by $\pi_2$ since $A_2 \not \cong \PP^1_k$.
        By \cite[Lemma 16.6]{KM}, $T_2$ is not a net. 
        Since we assume $x_1 \not \in A_1$ at the beginning of the proof, $(S_2, A_2+B_2)$ is not a banana.
        Since $a_2+b_2>1$ as we checked in (2), Proposition \ref{general-strategy} shows that $(S_2, A_2+B_2)$ is a fence or a tacnode.

        \item In Case 1 of the proof of \cite[Proposition 4.19]{Lac}, it is claimed that $(K_{S_2}+\frac23 A_2+B_2 \cdot B_2)>0$ because there are no tigers, but this seems to be wrong.
        Indeed, we only know that $a_2>a_1=\frac23$ and $(S_2, a_2A_2+b_2B_2)$ has no tigers, which only implies that $(K_{S_2}+a_2 A_2+B_2 \cdot B_2)>0$. 
        We will treat this case in Proposition \ref{prop:node-tacnode} to conclude that there is at most one possibility such that $A_1$ has a simple node and $(S_2, A_2+B_2)$ is a tacnode, which is Example \ref{eg:cexLac4.19}.
    \end{enumerate}
\end{rem}

\begin{eg}\label{eg:cexLac4.19}
In $\Char(k) \neq 2$, take a nodal cubic curve $A \subset \PP^2_k$.
Choose an inflection point $s \in A$ and let $G$ be the tangent line at $s$.
Then there is the other point $t \in A$ such that the tangent line $F$ of $A$ at $t$ also passes through $s$.
Indeed, by blowing up $\PP^2_k$ at $s$, the strict transform of $A$ becomes a bisection, and the existence of $t$ follows from the Riemann-Hurwitz formula.
(Here we use the assumption that $\Char(k) \neq 2$.)
By a similar argument, we can check that there are two points in $A$ at which the tangent line passes through $t$.
Let $u \in A$ be such a point and let $B$ be the tangent line at $u$.

First blow-up $\PP^2_k$ at $s$ three times along $A$ and at $t$ twice along $A$.
Next blow-up at $u$ three times along $B$.
Finally blow-up at the node of $A$ twice along one of the two branches of the node.
Then we obtain the minimal resolution $\wt S$ of a log del Pezzo surface $S$ of rank one such that $\Dyn(S)=[2,3,2^2]+[2,3,2]+[2^3]$.
Indeed, the strict transform of the $(-1)$-curve over $s$ in $S$ is of $(-K_S)$-degree $\frac{9}{44}$ and hence $-K_S$ is ample.

This example is not contained in \cite[\S 4.4]{Lac} because $A_1 \subset S_1$ has a node, $x_1 \not \in A_1$, and $(S_2, A_2+B_2)$ is a tacnode of genus two.

We note that such a log del Pezzo surface is log liftable. 
Indeed, we can write down the equations of curves $A$, $G$, $F$, $B \subset \PP^2_k$ regardless of $\Char(k)$ by virtue of \cite[Theorem 3.1]{LPS}.
However, we do not know whether this surface has a tiger in its minimal resolution or not.
\end{eg}

Note that $S=S(2A_1+A_3)$ satisfies $\pi_1^{\mathrm{alg}}(S^\circ) \neq 0$ in $\Char(k)=0$. 
Hence Example \ref{eg:cexLac4.19} does not contradict the content of \cite[\S 17]{KM}. 
Now let us show the following.

\begin{prop}\label{prop:node-tacnode}
    Suppose that $A_1$ has a simple node. 
    \begin{enumerate}
        \item[\textup{(1)}] If $(S_2, A_2+B_2)$ is a tacnode, then $S_0$ is constructed as in Example \ref{eg:cexLac4.19}.
        \item[\textup{(2)}] If $(S_2, A_2+B_2)$ is a fence, then one of \cite[Proposition 4.19 (1)--(3)]{Lac} holds.
    \end{enumerate}
\end{prop}

\begin{proof}
The assertion (2) follows by the same method as in the proof of \cite[Proposition 4.19]{Lac}.
Hence we may assume that $(S_2, A_2+B_2)$ is a tacnode.
As in Case 2 of the proof of [\textit{ibid}.], we obtain that $z \in \Sigma_2$, $a_1 = \frac23$, $g=2$, $r=1$, $E_1^2=-3$, and $K_{S_2}^2=4$.
In particular, $x_0$ is a $[2,3,2^2]$-singularity.
Since the first hunt step is an isomorphism at $x_1$, we have $e_1 \leq e_0$.
Since $(K_{S_2}+B_2 \cdot B_2)>-a_2(A_2 \cdot B_2)>-2$, we have $B_2 \not \subset S_2^\circ$.
Since $K_{S_2}+B_2$ is plt at singular points, $S_2$ contains a chain singularity.
Hence \cite[Theorem B.6]{Lac} shows that $S_2=S(2A_1+A_3)$.

Assume that $x_1$ is non-chain.
Then $B_2$ must pass through $2[2]+[2^3]$-singularities on $B_2$.
Since $x_1$ is non-Du Val, \cite[Proposition 10.1 (2)(d)]{KM} shows that $e_1 \geq \frac23=a_1$, a contradiction with flushness.
Hence $x_1$ is a chain singularity.

Since $(B_2 \cdot -K_{S_2})=(B_2 \cdot A_2)=g=2$, $B_2$ is a $0$-curve in $\wt S_2$ and hence $E_2^2=-3$.
Since $B_2 \not \subset S_2^\circ$, we have $x_1=[3, 2]$, $[3, 2^3]$, $[2,3,2]$, or $[2,3,2^3]$.
In the first case, we obtain $\frac25 = e_1 < e(z=[2^2], K_{S_1}+a_1A_1)=\frac49$, a contradiction with the maximality of $e_1$.
In the last case, we obtain $\frac47 = e_1 > e_0 = e([2,3,2^2])=\frac{6}{11}$, a contradiction with the maximality of $e_0$.

Assume that $x_1=[3, 2^3]$. Note that $K_{S_2}+B_2$ is plt. Let $Z \to S_2$ be the extraction of the $(-2)$-curves adjacent to $B_2$.
Then the singularities of $Z$ consist of $[2^2]+2[2]$-singularities.
On the other hand, $B_{2,Z}$ defines a $\PP^1_k$-fibration $Z \to \PP^1_k$ of relative Picard rank one.
By \cite[Lemma 3.4]{KM}, $Z$ cannot contain a $[2^2]$-singularity, a contradiction.
Hence $x_1=[2,3,2]$.

Let $Z \to S_2$ be the extraction of two $(-2)$-curves, say $G_1$ and $G_2$, over the $2[2]$-singularities.
Then the strict transform of $B_{2,Z}$ defines a $\PP^1_k$-fibration $Z \to \PP^1_k$ with a unique reducible fiber, which has exactly two components $F_1+F_2$.
By \cite[Lemma B.10]{Lac}, $F_1$ and $F_2$ are $(-1)$-curves and $F_1 \cap F_2$ is the $[2^3]$-singularity. 
We may assume that $(G_1 \cdot F_1)=(G_2 \cdot F_2)=1$.
Let $F_3$ be the $(-2)$-curve over the $[2^3]$-singularity adjacent to $F_2$.
Let $f \colon \wt S_2 \to \PP^2_k$ be the contraction of the strict transform of $G_2$ and the components of the total transform of $F_1+F_2$ except for $F_3$.
Then $f(A_{2, \wt S_2})$ is a nodal cubic curve, $f(G_{1,\wt S_2})$ is the tangent line of $f(A_{2, \wt S_2})$ at an inflection point $s=f(F_{1, \wt S_2})$.
The curve $f(F_{3, \wt S_2})$ is the tangent line at $t=f(F_{2,\wt S_2})$ and passes through $s$.
Moreover, $f(B_{2,\wt S_2})$ is the tangent line at another point $u \in f(A_{2, \wt S_2})$ and passes through $t$.
Therefore $S_0$ is constructed as in Example \ref{eg:cexLac4.19}.
\end{proof}

Note that log del Pezzo surfaces as in \cite[Proposition 4.19 (1) and (3)(a)]{Lac} are classified in \cite[\S 17.7--17.14]{KM} with one exception, which is written in \cite[Lemma 4.20]{Lac}.
We conclude that the assertions in \cite[\S 4.4]{Lac} still hold after replacing \cite[Proposition 4.19]{Lac} with Proposition \ref{prop:node-tacnode}.

\subsection{Smooth fences}\label{LacS4.5}
\cite[\S 4.5]{Lac} is devoted to dlt pairs $(S, X+Y)$ such that $S$ is a rank one log del Pezzo surface, and $X$ and $Y$ are smooth rational curves on $S$ such that $(X, Y)=1$.
When $\Char(k) = 0$, such a pair is analyzed by applying mainly \cite[Proposition 13.3]{KM}, which relies on deformation theory and the Bogomolov bound.
For that reason, \cite[\S 4.5]{Lac} contains characteristic-free results that can serve as a substitute for \cite[Proposition 13.3]{KM}.

\begin{rem} 
        \cite[Lemma 4.21]{Lac} still holds for a pair $(S, X+Y)$ such that $S$ is a rank one log del Pezzo surface, and $X$ and $Y$ are curves on $S$ such that $(X, Y)=1$.
        Indeed, we use this lemma in \cite[Lemmas 4.39 and 4.40]{Lac} for $(S_2, A_2+B_2)$ with $B_2$ singular.
\end{rem}

\subsection{$A_1$ is smooth}\label{LacS4.6}
\cite[\S 4.6]{Lac} is devoted to the case where $A_1$ is smooth.
In $\Char(k)=0$, such a log del Pezzo surface $S_0$ is classified in \cite[\S 18-19]{KM}.
Since the Bogomolov bound is used ubiquitously in \cite[\S 18]{KM}, Lacini rewrites the proof entirely to handle this case.
The hardest part is to show that $(S_2, A_2+B_2)$ is not a fence.
To prove this, he makes use of the contents of \cite[\S 4.5]{Lac}.
By \cite[Lemmas 4.34 and 4.40]{Lac} and \cite[Lemmas 18.7 and 18.8]{KM}, we conclude that $(S_2, A_2+B_2)$ is a (smooth) banana.

Note that \cite[\S 19]{KM} is based on \cite[Proposition 13.2]{KM} whose proof relies on the Bogomolov bound. 
Lacini avoids the use of the Bogomolov bound as follows.
Since $(S_2, A_2+B_2)$ is a banana, it holds that $x_1 \in A_1$, $B_2$ contains at most one singular point of $S_2$, and $\pi_2$ is as listed in \cite[Lemma 11.1.1]{KM}.
By \cite[Lemma 4.33]{Lac}, we have $a_2+b_2>a_1+e_1>1$.
Hence $\pi_2$ is not in configuration I.
If $B_2 \subset S_2^\circ$, then $B_2$ is a tiger of $(S_2, a_2 A_2 + b_2 B_2)$, a contradiction.
In particular, both $E_2$ and $B_2$ pass through exactly one singular point and $\Sigma_2$ must intersect with $E_2$ at a smooth point of $T_2$.
Since $\pi_2$ is not in configuration I, $\Sigma_2$ intersects with $A_2$ at a singular point.
If $A_1$ passes through four singular points, then $(S_2, A_2+B_2)$ is a smooth banana such that $A_2$ and $B_2$ contain exactly two and one singular points respectively, a contradiction with \cite[Lemma 4.41]{Lac}.
Therefore $A_1$ passes through three singular points, and $(S_2, A_2+B_2)$ is a smooth banana such that both $A_2$ and $B_2$ contain exactly one singular point.
The rest of the arguments is the same as in the proof of \cite[Proposition 13.2]{KM}.
As a result, we obtain the list as in \cite[\S 19]{KM}.

Note that there is at most one log del Pezzo surface erroneously omitted in the list \cite[19.4.1.2]{KM} (see Remark \ref{rem:KM19.4.1.2} for more detail).
The assumption $\Char(k) \neq 2,3$ is used only in the following.
To prove \cite[Lemma 4.39 (4) and (5)]{Lac}, we need to use \cite[Lemma B.9]{Lac} (see also Remark \ref{rem:Section4.6} (4)).
In the proof of \cite[Lemma 18.7]{KM}, we use \cite[Lemma B.9]{Lac} (see also Lemma \ref{lem:sigma2}).
In the proof of \cite[Lemma 18.8]{KM}, we use the Riemann-Hurwitz formula for 2-sections and 3-sections.

\begin{rem} \label{rem:Section4.6}
Let us mention some remarks on the proofs in \cite[\S 4.6]{Lac}.
    \begin{enumerate}
        \item In Case 1 of the proof of \cite[Lemma 4.33]{Lac}, we do not have to use the Riemann-Hurwitz formula.
        Indeed, \cite[Lemma 11.5.9]{KM} shows that a triple fiber cannot contain $[3,2]+[2]$-singularities.
        As a result, \cite[Lemma 4.33]{Lac} holds in $\Char(k) \geq 0$.
        
        \item Since I cannot follow why $(K_{S_1}+A_1) \cdot A_1 <\frac23$ in Case 2 of the proof of \cite[Lemma 4.34]{Lac}, I give another proof as follows.
        Suppose that $A_1$ is smooth, $(S_2, A_2+B_2)$ is a tacnode, and $A_1$ contains four singular points.
        As in the proof of [\textit{ibid}.] before Case 1, we obtain that
        \begin{enumerate}
            \item $S_1$ is Du Val along $A_1$,
            \item $(K_{S_1}+A_1) \cdot A_1 \geq \frac13$,
            \item $a_1<\frac34$, and
            \item $\pi_2$ is in configuration II with $g=2$ such that the $[2^g]$-singularity lies on $A_1$.
        \end{enumerate}
        Take $s,t,u$ so that $A_1$ contains $[2^s]+[2^t]+[2^u]+[2^2]$-singularities.
        We have $E_1^2 \leq -3$ since otherwise $x_0$ is Du Val.
        Since $e([3,[2],[2],[2]])=\frac23$ and $e_0<a_1<\frac34$, we obtain $x_0=[3,[2],[2],[2]]$ by \cite[Lemmas 8.3.9 and L.1]{KM}.
        In particular, $s=t=u=1$.
        Then $(K_{S_1}+A_1 \cdot A_1)=-2+3 \cdot \frac12+\frac23=\frac16$, a contradiction with (b).
        
        \item In Case 1b of the proof of \cite[Lemma 4.36]{Lac}, the observation on branches only gives us that $x_0=[2;[2],[2,2],[3]]$ or $[3,[2],[2],[2]]$. Thus we have to deny the latter case.
        Suppose that $x_0 = [3,[2],[2],[2]]$. We obtain $\alpha=-(K_{S_2} + A_2 \cdot A_2) = 2-2 \cdot \frac12-\frac{s-1}{s}>0$, where $s$ is the index of the fourth singularity on $A_1$.
        On the other hand, we have shown that $\alpha \leq 0$ in the beginning of Case 1 of [\textit{ibid}.], a contradiction.
        Consequently, the statement of [\textit{ibid}.] still holds.
        
        \item In the proof of \cite[Lemma 4.39 (4)]{Lac}, the condition $\frac{1}{1-\alpha} \in \Z$ does not imply $\alpha=0$ directly since we do not know whether $\alpha \in \Z$.
        We will treat this case in Proposition \ref{prop:fence-nonDV} to conclude that [\textit{ibid}.] still holds.
        We also show in Lemma \ref{lem:Lac4.39-Char23} and Example \ref{eg:Lac4.39-char3} that the assumption $\Char(k) \geq 5$ is essential.
        
        \item In Cases 2b and 4a in the proof of \cite[Lemma 4.40]{Lac}, it seems that we can obtain only $j \geq 2$ instead of $j \geq 3$ since 
        $e([2^{j+1}], K_{S_1}+a_1A_1) =\frac{j+1}{j+2} \geq \frac13a_1+\frac13= e([3]\in A_1, K_{S_1}+a_1A_1)$ and $a_1>\frac23$ by \cite[Lemma 4.33]{Lac}.
        In Case 4a, the equation $\frac{1}{x}+\frac1y=\frac{6j+6}{10j+9}$ still has no solution.
        On the other hand, in Case 2b, a new possibility arises: $A_2$ passes through either $2[3]$, $[3]+[2]$, or $[3]+[2,2]$-singular points.
        However, in any case, we obtain that $A_2^2 \leq 0$ since $(A_{2, \wt S_2})^2=-1$, a contradiction with the ampleness of $A_2$. 
        Consequently, the statement of [\textit{ibid}.] still holds.    
    \end{enumerate}
\end{rem}

Let us modify the proof of \cite[Lemma 4.39 (4)]{Lac}.

\begin{lem}[cf. {\cite[Lemma 18.3]{KM}}]\label{KM18.3}
Suppose that $A_1$ is smooth and $(A_{1, \wt S_1})^2=-1$.
Then $S_1$ is not Gorenstein.
\end{lem}

\begin{proof}
Suppose that $S_1$ is Gorenstein.
Since $x_0$ is a klt singularity, $A_1$ contains at most four singularities.
Let $[2^s]$, $[2^t]$, $[2^u]$, $[2^v]$ be the singularities lying on $A_1$ with $s \geq t \geq u \geq v \geq 0$.
Since $(K_{S_1}+A_1 \cdot A_1)>0$, we have $\frac{1}{s+1}+\frac{1}{t+1}+\frac{1}{u+1}+\frac{1}{v+1}<2$.
Comparing with the classification of Du Val del Pezzo surfaces \cite[Theorem 1.1]{KN2}, we obtain that $(s,t,u,v)=(2,2,2,2)$ or $(3,3,1,1)$.
Hence $K_{S_1}^2=1$ and $S_1=S(4A_2)$ or $S(2A_1+2A_3)$.

On the other hand, let $Z$ be the blow-up of $\wt S_1$ at $\Bs|-K_{\wt S_1}|$, which is an extremal rational (quasi-)elliptic surface, and let $G \subset Z$ be the exceptional divisor over $\wt S_1$.
Then both $G$ and the strict transform of $A_1$ are sections such that for each reducible fiber of $Z$ they intersect with the same irreducible component.
However, we can check that there is no such choice of sections by \cite{M-P, Lang1, Lang2, Ito1, Ito2}, a contradiction.
\end{proof}

\begin{prop}[cf. {\cite[Lemma 4.39 (4)]{Lac}}]\label{prop:fence-nonDV}
Suppose that $A_1$ is smooth and $(S_2, A_2+B_2)$ is a fence.
Then $A_2$ contains at least one non-Du Val point unless $B_2$ is a cuspidal rational curve contained in $S_2^\circ$ and $\Char(k) = 2, 3$.
\end{prop}

\begin{proof}
    As in the proof of {\cite[Lemma 4.39]{Lac}}, we obtain the following:
    \begin{enumerate}
        \item $x_1 \in A_1$.
        \item $B_2$ is singular and contains at most one singular point of $S_2$.
        \item $(A_{2, \wt S_2})^2=-1$.
        \item $A_2$ is in the Du Val locus if and only if $B_2$ has genus one and is in the smooth locus.
    \end{enumerate}
    Note that assertions (1)--(3) show that $(A_{1, \wt S_1})^2=(A_{2, \wt S_2})^2=-1$. 
    We may assume that $B_2$ has genus one and is in the smooth locus.
    In particular, $S_2$ is Du Val.
    
    Assume in addition that $B_2$ has a node.
    Then $\pi_2$ is not in configuration I since \cite[Lemma 4.33]{Lac} shows that $b_2>e_1>\frac12$. 
    Since $E_2$ contains at most one singularity, $\pi_2$ must be in configuration $(II, x^{m-1})$ for some $m \geq 1$ and $b_2=\frac{m+1}{m+2}$.
    In particular, $T_2$ is Du Val.
    On the other hand, $S_1$ is not Du Val by Lemma \ref{KM18.3}.
    Thus $x_1=[j,2^m]$ for some $j \geq 3$.
    By \cite[Lemma 8.0.8]{KM}, the spectral value of $(S_1, a_1A_1)$ at $x_1$ is $(j-2)(m+1) \geq m+1$.
    However, \cite[Lemma 8.0.7(2)]{KM} now shows that $\frac{m+1}{m+2}=b_2>e_1 \geq \frac{m+1}{m+2}$, a contradiction.
    
    Hence $B_2$ has a cusp. Suppose in addition that $\Char(k) \neq 2, 3$.
    Since $K_{S_2}+A_2$ is plt, $S_2$ is the Du Val del Pezzo surface of type $[2]+[2^2]$ by \cite[Lemma B.10]{Lac}.
    Since $B_2^2=6$, $\pi_2$ is factored by the blow-up $\pi'$ at the cusp of $B_2$ three times along $B_2$.
    In particular, the $\pi'$-exceptional $(-1)$-curve $C$ is contained in $\wt S_0$.
    If $b_2 \leq \frac56$ (resp.\ $\geq \frac56$), then $A_2$ (resp.\ $C$) is a tiger of $(S_2, a_2A_2+b_2B_2)$, a contradiction.
\end{proof}

We can determine when $A_2$ of a fence $(S_2, A_2+B_2)$ is contained in the Du Val locus as follows.

\begin{lem}\label{lem:Lac4.39-Char23}
Suppose that $A_1$ is smooth, $(S_2, A_2+B_2)$ is a fence, and $A_2$ is in the Du Val locus.
Then the following holds:
\begin{enumerate}
    \item[\textup{(1)}] $B_2$ is a cuspidal rational curve in $S_2^\circ$.
    \item[\textup{(2)}] $A_2$ passes through exactly two singularities.
    \item[\textup{(3)}] If $\Char(k)=3$, then $S_2 = S(3A_2)$ or $S(A_2+A_5)$.
    \item[\textup{(4)}] If $\Char(k)=2$, then 
    $S_2=S(2A_1+A_3)$ or $S(A_1+A_5)$.
    In the former case, $A_2$ passes through $[2]+[2^3]$-singularities.
\end{enumerate}
\end{lem}

\begin{proof}
The assertion (1) follows from Proposition \ref{prop:fence-nonDV}.
Furthermore, $S_2$ is Du Val.
Set $n=B_2^2=(-K_{S_2})^2 \in \Z_{>0}$.
Note that $\sharp(A_2 \cap \Sing(S_2))=\sharp(A_1 \cap \Sing(S_1))-1=2$ or $3$.
Take $s \geq t \geq u \geq 0$ such that $A_2$ passes through $[2^s]+[2^t]+[2^u]$-singularities.

By \cite[Lemma 4.21]{Lac}, we have 
\begin{align*}
    \frac1n&=\frac{1}{B_2^2}=\frac{1+(K_{S_2}+A_2 \cdot A_2)}{1+(K_{S_2}+B_2 \cdot B_2)}=1+(K_{S_2}+A_2 \cdot A_2)\\
    &=2-\left(\frac{1}{s+1}+\frac{1}{t+1}+\frac{1}{u+1}\right).
\end{align*}

\noindent (2):
It suffices to show that $u=0$.
Suppose by contradiction that $u >0$.
Then we can check that
\begin{align}\label{eq:sing-candidate}
    (n,s,t,u)=(1,2,2,2), (1,3,3,1), (1,5,2,1), \text{ or } (2,1,1,1).
\end{align} 
By \cite[Lemma 3.2]{KN2}, there is a birational morphism $Z \to \wt S_2$ from an extremal rational (quasi-)elliptic surface $Z$.
As in the proof of \cite[Lemma B.9]{Lac}, we can choose $Z$ so that it contains a singular fiber of type $II$.

Suppose in addition that $\Char(k) = 3$.
Then, by comparing with (\ref{eq:sing-candidate}), an analysis similar to \cite[\S3.3]{KN2} shows that $(n,s,t,u)=(1,2,2,2)$ and $Z$ is the quasi-elliptic surface of type (3) as in \cite[Theorem 3.1]{Ito1}.
Hence $S_2=S(4A_2)$.
Since $B_2^2=n=1$, we have $E_2^2 \leq -3$.
Since $A_1$ passes through exactly $3[2^2]$-singularities and $x_1$, whose index is at least three, $x_0$ cannot be a klt singularity, a contradiction.

Hence $\Char(k) =2$.
By comparing with (\ref{eq:sing-candidate}), an analysis similar to \cite[\S3.2]{KN2} shows that $(n,s,t,u)=(2,1,1,1)$ and we can choose $Z$ as a quasi-elliptic surface of type (e) or (f) as in \cite[Proposition 5.1]{Ito2}.
Hence $S_2=S(3A_1+D_4)$ or $S(7A_1)$.
In each case $A_2$ passes through exactly $3[2]$-singularities.
In particular, $E_2^2\leq n-4=-2$.
If $x_1$ is a $[2]$-singularity, then $A_1$ passes through exactly $4[2]$-singularities, a contradiction with $(K_{S_1}+A_1 \cdot A_1)>0$. 
Hence $E_2^2 \leq -3$.
If the equality holds, then $x_1=[3,2]$.

Suppose in addition that the center $q_1$ of $\pi_1$ is distinct from $x_1$.
Then $q_1$ is a $[2]$-singularity.
In particular, $x_0=[2+k; [2], [2], x_1]$ for some $k \geq 0$.
Furthermore, $\Sigma_{1, \wt S_0}$ intersects with the central curve of $x_0$ and the $(-2)$-curve (or the $(-3)$-curve when $k=0$) corresponding to the end of a $[2^k,3]$-chain.
Let $r$ and $s$ be the index and the spectral value of $x_1$.
Since either $E_2^2 \leq -4$ or $E_2^2=-3$ with $x_1=[3,2]$, we have $r \geq 4$ and $s \geq 2$ by \cite[Lemma 8.0.8]{KM}.
By \cite[Lemma 8.3.9]{KM}, we obtain $e_0=\frac{kr+s}{kr+s+1}$.
Since $(K_{S_0} \cdot \overline{\Sigma}_1)=-1+\frac{1}{2k+3}+e_0<0$, we have $kr+s+1<2k+3$.
Hence $0 \leq k<\frac{2-s}{r-2} \leq \frac{0}{r-2}$, a contradiction.

Hence $q_1=x_1$.
Set $j=-E_2^2$ and define a chain $x_1'$ so that $x_1=[j, x_1']$.
Then $x_0=[2+k; [2], [2], [2]]$ for some $k \geq 1$.
Furthermore, $\Sigma_{1, \wt S_0}$ intersects with the central curve of $x_0$ and the $(-2)$-curve, say $E_1'$, corresponding to the end of a $[2^k,j+1,x_1']$-chain.
Since $(K_{S_0} \cdot \overline{\Sigma}_1)=-1+e_0+e(E_1', S_0)<0$ and $e_0=\frac{2k}{2k+1}$, 
it holds that $e(E_1', S_0) <\frac{1}{2k+1}$.
By \cite[Lemma L.1]{KM}, we obtain 
\begin{align*}
e([2^k,j+1])=\frac{(k+1)(j-1)}{j(k+1)+1}&<e(E_2,S_0)\\&<(k+1) \cdot e(E_1',S_0)<\frac{k+1}{2k+1}.    
\end{align*}
Simplifying the equation yields $j < 2+\frac{2}{k}$.
Since $j=-E_2^2 \geq 3$, we obtain $k=1$, $j=3$, and $x_1=[3,2]$.
In particular, $e_0=e([3; [2], [2], [2]])=\frac23$.
Since $E_1'$ is the end of a $[2,4,2]$-chain, we have $e(E_1', S_0)=\frac13$.
Therefore $(K_{S_0} \cdot \overline{\Sigma}_1)=-1+e_0+e(E_1', S_0)=0$, a contradiction.

Combining these results, we obtain the assertion.

\noindent (3) and (4): By the assertion (2), we have $u=0$ and we can check that 
\begin{align}\label{eq:sing-candidate2}
  (n,s,t)=(2,3,3), (2,5,2), (3,2,2), (3,5,1), (4,3,1), \text{ or } (6,2,1).  
\end{align}
In the last case, we have $S_2 = S(A_1+A_2)$ and can show the existence of a tiger as in the proof of Proposition \ref{prop:fence-nonDV}, a contradiction. 
The rest of the proof runs as before.
Indeed, by comparing with (\ref{eq:sing-candidate2}), an analysis similar to \cite[\S\S 3.2--3.3]{KN2} determines $\Char(k)$ and $\Dyn(S_2)$.
In particular, the first case of (\ref{eq:sing-candidate2}) cannot occur.
\end{proof}

The following examples show that the assumption $\Char(k) \geq 5$ is essential in \cite[Lemma 4.39 (4)]{Lac} \textit{a priori}.
However, we do not know whether such a surface has a tiger in its minimal resolution or not.

\begin{eg}\label{eg:Lac4.39-char3}
In $\Char(k)=3$, let $S_2$ be the Du Val del Pezzo surface of type $3[2^2]$.
Since $S_2$ is constructed from a quasi-elliptic surface of type (2) as in \cite[Theorem 3.1]{Ito1} (see \cite[\S 3.3]{KN2} for more detail), there is a cuspidal anti-canonical member $B_2$ in $S_2^\circ$.
Let $A_2 \subset S_2$ be a $(-1)$-curve, which passes through two of the three $[2^2]$-singularities such that $K_{S_2}+A_2$ is plt.
Next blow-up at the cusp of $B_2$ three times along $B_2$.
Next blow-up at $A_2 \cap B_2$ twice along $A_2$.
Finally contract all $K$-non-negative curves.
Then we obtain a log del Pezzo surface $S$ of rank one with $\Dyn(S)=[2^2,3,2^2]+[2,4]+[3]+[2^2]+[2]$.
Indeed, the strict transform of the $(-1)$-curve over $A_2 \cap B_2$ in $S$ is of $(-K_S)$-degree $\frac{4}{35}$ and hence $-K_S$ is ample.

Suppose that $S_0=S$ has no tigers. 
Then $x_0$ is the $[2^2,3,2^2]$-singularity and $A_1$ is smooth. 
Since $\Dyn(S_1)=2[2^2]+2[3]+[2]$ and $A_1$ passes through exactly $2[2^2]+[3]$-singular points, the point $x_1$ must be the $[3]$-singular point on $A_1$.
Then $(S_2, A_2+B_2)$ is a fence and $A_2$ is contained in the Du Val locus of $S_2$. In particular, $\alpha=-(K_{S_2}+A_2 \cdot A_2)=\frac23>0$.
\end{eg}

The original proof of \cite[Lemma 18.7]{KM} seems to use \cite[Proposition 13.4]{KM} whose proof relies on the Bogomolov bound.
For the convenience of the reader, we prove Lemma \ref{lem:sigma2}, which corresponds to \cite[Lemma 18.7]{KM}.

\begin{lem}[{\cite[Lemma 18.7]{KM}}]\label{lem:sigma2}
Suppose that $\Char(k) \geq 5$ and $A_1$ is smooth.
Then $A_1 \neq \Sigma_2$.
\end{lem}

\begin{proof}
Suppose by contradiction that $A_1=\Sigma_2$.
Then as in the proof of \cite[Lemma 18.7]{KM}, it holds that $B_2$ has a simple cusp at $q_2$ and $\pi_2$ is in configuration III.
If $B_2 \subset S_2^\circ$, then $B_2$ is a tiger of $(S_2, b_2B_2)$, a contradiction with Lemma \ref{lem:Sarkisov} (8).
Hence $B_2$ contains exactly one singularity.
Note that $b_2>e_1>\frac12$ by \cite[Lemma 4.33]{Lac}, and $-(K_{S_2}+\frac12B_2)$ is ample by Lemma \ref{lem:Sarkisov}.

Now we follow the notation of \cite[Lemma-Definition 12.4]{KM} with $S=S_2$.
Consider the morphism $f \colon Y \to S_2$ extracting the exceptional divisor $G$ adjacent to $B_2$ and the morphism $\pi \colon Y \to W$ contracting $M$.
Then, as in Case 2 of the proof of \cite[Lemma 4.12]{Lac}, it holds that $B_{2,W} \subset W^\circ$, $G_W$ is a smooth rational curve of anti-canonical degree one, and
$W$ is listed in \cite[Lemma B.9]{Lac} (see also Remark \ref{rem:Section4.3} (2)). 
By Lemma \ref{lem:Sarkisov}, there are $b>b_2$ and $c>0$ such that $(W, bB_{2,W}+cG_W)$ has no tigers in $\wt S_0$.

Suppose that $W = S(A_1+A_2)$.
If $b \geq \frac56$ (resp.\ $b <\frac56$), then $A_1=\Sigma_2$ (resp.\ $G_W$) is a tiger of $(W, bB_{2,W}+cG_W)$ in $\wt S_0$, a contradiction.

Suppose that $W=S(E_8)$, $S(E_7)$, $S(E_6)$, $S(D_5)$, or $S(A_4)$.
Then, as in the last part of the proof of \cite[Lemma 4.12]{Lac}, we can reduce to the case where $W = S(A_1+A_2)$, a contradiction.
Hence $W=S(2A_4)$.

An analysis similar to Remark \ref{rem:Section4.3} (4) gives that $B_2$ contains a $[2+r, 2^4]$-singularity for some $r \geq 0$.
Since $\pi_2$ is in configuration III, it holds that $x_1$ is a $[5,2+r, 2^4]$-singularity, whose spectral value is at least $18$ by \cite[Lemma 8.0.8]{KM}.
By \cite[Lemma 8.0.9]{KM}, we conclude that $b_2>\frac{18}{19}>\frac56$ and $A_1=\Sigma_2$ is a tiger of $(S_2, b_2B_2)$, a contradiction.
\end{proof}

\begin{rem} \label{rem:KM19.4.1.2}
        In \cite[19.4.1.2]{KM}, one possibility seems to be erroneously omitted.
        When $(3,2,2)$ is added and $z=[2]$, then $K_{S_0}$ is anti-ample if $k=3$.
        
        Indeed, the coefficient of the $(-3)$-curve in the $[3',2^{k-2},4,2^2]$-singularity equals $\frac{7k+2}{14k-1}$ and the coefficient of $z=[2]$ in $x_0$ equals $\frac{e_0}{2}=\frac{5k-10}{10k-18}$.
        Hence $(-K_{S_0} \cdot \overline{\Sigma}_0)=1-\frac{7k+2}{14k-1}-\frac{5k-10}{10k-18}=\frac{11(4-k)}{(14k-1)(10k-18)}$, which is positive if $k = 3$.
\end{rem}

\section{Plt pairs}\label{LacS5.1}
\cite[\S 5.1]{Lac} is devoted to dlt pairs $(S, C)$ such that $S$ is a log del Pezzo surface of rank one and $C$ is a reduced irreducible curve such that $K_S+C$ is anti-nef.
Taking into account the fact that the smoothness of $C$ is implicitly assumed in \cite[Lemma 5.4]{Lac}, we are only dealing with plt pairs among them.
These pairs appear in Theorem \ref{Lac6.2} (1)(c) and (2).
To classify such a pair $(S, C)$, we select a hunt based on the number of singularities on $C$.
For the convenience of the reader we repeat the scaling convention used in \cite[\S 5.1]{Lac}.

\begin{defn and notation}\label{DaN}
Let $(S, C)$ be a plt pair of a rank one log del Pezzo surface $S$ and a reduced irreducible curve $C$ such that $K_S+C$ is anti-nef.
    \begin{enumerate}
        \item[\textup{Case 1}:] Suppose that $C$ contains at least three singular points. Then we take $1 - \epsilon < a < 1$ for sufficiently small $\epsilon>0$ and take $f_0 \colon T_1 \to S$ and $\pi_1$ as a hunt step for $(S, aC)$.
        Set $b_1$ such that $K_{T_1}+aC_{T_1}+b_1E_1$ is $R$-trivial, where $R$ is the edge of $\overline{\mathrm{NE}}(T)$ defining $\pi_1$. 
        \item[\textup{Case 2}:] Suppose that $C$ contains at most two singular points. 
        Then we take $f_0 \colon T_1 \to S$ and $\pi_1$ as a hunt step for $(S, C)$.
        Set $b_1$ such that $K_{T_1}+C_{T_1}+b_1E_1$ is $R$-trivial, where $R$ is the edge of $\overline{\mathrm{NE}}(T)$ defining $\pi_1$. 
    \end{enumerate}
If $\pi_1 \colon T_1 \to S_1$ is birational, then we define $A_1$ as $E_{1, S_1}$ and $\Sigma_1$ as the exceptional curve of $\pi_1$.
If $C_{T_1} \neq \Sigma_1$ in addition, then we write $C_1 \coloneqq C_{S_1}$.
\end{defn and notation}

\begin{rem} 
        In Case 1, we can pick $E_1$ as in the proof of \cite[Proposition 23.5]{KM}. 
        Hence, when $\pi_1$ is birational, we can show the inequality $b_1 < a$ in much the same way as [\textit{ibid}.].
\end{rem}

For each scaled hunt, we can control the output as follows.

\begin{lem}\label{lem:dlthunt}
Let $(S, C)$ be a plt pair of a rank one log del Pezzo surface $S$ and a reduced irreducible curve $C$ such that $K_S+C$ is anti-nef.
Suppose that $C \cap \Sing S \neq \emptyset$.
Run the hunt as in Cases 1 or 2 of Definition and Notation \ref{DaN}.
Then one of the following holds.
\begin{enumerate}
    \item[\textup{(1)}] $T_1$ is a net.
    \item[\textup{(2)}] $C_{T_1}$ is contracted by $\pi_1$ (that is $C_{T_1} = \Sigma_1$).
    \item[\textup{(3)}] $C_{T_1}$ is not contracted by $\pi_1$ and $(S_1, C_1+A_1)$ is dlt at $\Sing S_1$.
\end{enumerate}
If (3) occurs, then the following holds.
\begin{enumerate}
    \item[\textup{(3-1)}] If $x_0 \not \in C$, then $(S_1, C_1+A_1)$ is a fence.
    \item[\textup{(3-2)}] If $x_0 \in C$, then $(S_1, C_1+A_1)$ is a fence or a banana.
\end{enumerate}
\end{lem}

\begin{proof}
In what follows, we may assume that $T_1$ is not a net and $C_{T_1} \neq \Sigma_1$.
Suppose that we run the hunt as in Case 1.

\noindent (3): Since $b_1<a<1$, we can show that $(T_1, aC_{T_1}+b_1E_1)$ is flush as in the proof of \cite[8.4.1 (1)]{KM}. 
Then $(S_1, aC_1+b_1A_1)$ is also flush by \cite[Lemma 8.3.1 (2)]{KM}.
Therefore the assertion follows from \cite[Lemma 8.0.4]{KM}.

\noindent (3-1): Suppose in addition that $x_0 \not \in C$. 
Then $C_1 \cap A_1$ consists of one point $q_1=\pi_1(\Sigma_1)$.
By the assertion (3), we have $q_1 \in S_1^\circ$.
Since $(S_1, aC_1+b_1A_1)$ is flush, \cite[Lemma 8.3.7 (1)]{KM} shows that 
\begin{align*}
a \cdot \mult_{q_1}C_1+b_1 \cdot \mult_{q_1}A_1-1<b_1.
\end{align*}
Taking $\epsilon \to 0$, we obtain $a \to 1$ and
\begin{align*}
     \mult_{q_1}C_1+(\lim_{\epsilon \to 0} b_1) \cdot (\mult_{q_1}A_1-1) \leq 1.
\end{align*}
In particular, $\mult_{q_1}C_1=\mult_{q_1}A_1=1$.
If $q_1 \in A_1 \cap C_1$ is a node of genus $\geq 2$, then \cite[Lemma 8.3.7 (2)]{KM} shows that 
$2a+b_1<2$ and hence $2+(\lim_{\epsilon \to 0}b_1) \leq 2$,
a contradiction.
Hence $(S_1, C_1+A_1)$ is a fence.

\noindent (3-2): 
Suppose in addition that $x_0 \in C$.
If $q_1=\pi_1(\Sigma_1)$ is not contained in $C_1$, then $(S_1, C_1+A_1)$ is a fence.
If $q_1 \in C_1$, then as in the proof of the assertion (3-1), we obtain that $(S_1, C_1+A_1)$ is a banana.

Now suppose that we run the hunt as in Case 2.
If $b_1<1$, then we can show the flushness of $(S_1, C_1+b_1A_1)$ and the assertions (3), (3-1), (3-2) as in Case 1.

Now suppose in addition that $b_1 \geq 1$.
Then $(T_1, C_{T_1}+E_1)$ is dlt.
Since $(K_{T_1}+C_{T_1}+E_1 \cdot \Sigma_1) \leq 0$, we have $e(\Sigma_1, K_{S_1}+C_1+A_1) \leq 0$.
By \cite[Theorem 2.44]{KM98}, we conclude that $(S_1, C_1+A_1)$ is dlt, which proves the assertion (3).
Since $(S_1, C_1+A_1)$ is dlt, the assertions (3-1) and (3-2) hold.
\end{proof}

Based on Lemma \ref{lem:dlthunt}, we classify $(S, C)$ in \cite[Propositions 5.1-5.3]{Lac} and \cite[Lemma 5.4]{Lac}.
When $(S_1, C_1+A_1)$ is a banana, we mainly use \cite[Lemma 4.41]{Lac}.
If each of $C_1$ and $A_1$ contains a unique singular point in addition, then we also use \cite[Proposition 13.2]{KM}.
On the other hand, when $(S_1, C_1+A_1)$ is a fence, we use the results in \cite[\S 4.5]{Lac}.
We also use arguments as in the proof of \cite[Lemma 4.39 (3)]{Lac} instead of \cite[Proposition 13.3]{KM} since the latter relies on deformation theory and the Bogomolov bound.

The assumption $\Char(k) \neq 2,3$ is used only to obtain \cite[Proposition 5.1]{Lac}.
We use \cite[Lemma B.9]{Lac} to obtain the descriptions (2) and (5) of \cite[Proposition 5.1]{Lac}.
We also apply the Riemann-Hurwitz formula in the case where $T_1$ is a net and $E_1$ is a 2- or 3-section.  

\begin{rem}
Let us mention some remarks on the proofs in \cite[\S 5.1]{Lac}.
    \begin{enumerate}
        
        \item The list $S(A_1+A_2+A_5)$, $S(2A_1+A_3)$, $S(4A_1)$ in both \cite[Proposition 23.5 (4)]{KM} and \cite[Proposition 5.1]{Lac} should be replaced by the list $S(A_1+A_2+A_5)$, $S(2A_1+A_3)$, $S(4A_2)$.    

        \item In Case 1a of the proofs of \cite[Propositions 5.2 and 5.3]{Lac}, we can deny the case where $E_1$ is a bisection as follows.

        Assume that $E_1$ is a bisection.
        Then $b_1=\frac12$.
        Let $y$ be a singularity on $C$ and $r$ the index of $y$.
        Since $\frac12=b_1 \geq e(y; K_S+C)=\frac{r-1}{r}$, we obtain $r = 2$.
        Hence $y$ is a $[2]$-singularity.
        Since there are at most $2[2]$-singular points on $C$, we have $C_{\wt S}^2 \geq 0$.
        
        Let $\wt S \to T \to \PP^1_k$ be the induced morphism.
        Take $\wt S_1 \to H$ to be the MMP/$\PP^1_k$ isomorphic along $C$.
        Then $H$ is a Hirzebruch surface, $C_H$ is a section with $C_{H}^2=C_{\wt S_1}^2 \geq 0$, and $E_{1,H}$ is a bisection disjoint from $C_H$.
        However, there is no such pair of a section and a 2-section in any Hirzebruch surface, a contradiction.

        \item In Case 2a of the proof of \cite[Proposition 5.2]{Lac}, we can deny the case where $E_1$ is a bisection as follows.

        As above, we can check that there are $2[2]$-singular points on $C$ and $x_0$ is one of them.
        In particular, $E_1 \subset T_1^\circ$.
        Since $C_{T_1}$ is a section, $T_1$ contains a unique multiple fiber, say $F$, which passes through the $[2]$-singular point on $C_{T_1}$.
        Since $E_1$ is contained in the smooth locus, it intersects with $F$ at one point transversely, and $F$ is of multiplicity two.
        In particular, $F$ must be of type (5) with $k=1$ of \cite[Lemma 11.5.9]{KM}.

        Now take $S_1 \to H$ as above.
        Then we can check that $H$ is a Hirzebruch surface and $E_{1,H}$ is a bisection with $E_{1,H}^2=0$, a contradiction.
        
        \item Since the proof of \cite[Proposition 5.3]{Lac} is independent of \cite[Proposition 5.2]{Lac}, we can classify $(S, C)$ as in case (7) of [\textit{ibid}.].
        
        \item In \cite[Propositions 5.2 (5) and (7), and 5.3]{Lac}, the phrase ``$K$-positive'' should be replaced by ``$K$-nonnegative'' since we have to take $(-2)$-curves into account.
        
        \item In \cite[Proposition 5.3 (7)]{Lac}, $F \setminus (C \cup E)$ should be replaced by $F \setminus C$.
        Indeed the point can be $F \cap E$ as follows.

        Start with $\FF_2$.
        Pick the negative section $E$, a fiber $F$, and a tautological section $C$.
        Blow up at $F \cap E$ twice along $F$ and contract all the $K$-non-negative curves.
        Then we obtain a log del Pezzo surface with $\Dyn(S)=[2,3]+[2]$ such that $C_S \cap \Sing S$ is a $[2]$-singularity.
        Since $e([2,3])<\frac12$, the point $x_0$ must be the $[2]$-singularity on $C_S$ in the hunt step.
        
        \item In \cite[Lemma 5.4]{Lac}, we need to add the assumption that $C$ is a smooth rational curve.
    \end{enumerate}
\end{rem}

\section{Du Val del Pezzo surface}\label{DuVal}
In this section, we consider pairs $(S, C)$ such that $S$ is a Du Val del Pezzo surface of rank one and $C$ is an irreducible, reduced, and singular anti-canonical member of $S$ contained in $S^\circ$.
These pairs appear in Theorem \ref{Lac6.2} (1)(b).
When $C$ has a simple cusp, we obtain the following.

\begin{lem}[{\cite[Lemma B.9]{Lac}}]\label{lem:tiger-cusp}
    Let $\Char(k) \neq 2,3$, $S$ be a Du Val del Pezzo surface of rank one, and $C \in |-K_S|$ such that $C$ has a simple cusp and $C \subset S^\circ$.
    Then one of the following holds.
    \begin{enumerate}
        \item[\textup{(1)}] $S \cong \PP^2_k$ and $C$ is a cuspidal cubic.
        \item[\textup{(2)}] $S \cong S(A_1)$ and $C$ is a cuspidal quartic in $S(A_1) \subset \PP^3_k$.
        \item[\textup{(3)}] $S$ is obtained by blowing down $(-1)$-curves on an extremal rational elliptic surface $Z$ and $C_Z$ is a singular fiber of type $II$.
    \end{enumerate}
    In Case (3), one of the following holds.
    \begin{enumerate}
        \item[\textup{(3-1)}] The singular fibers of $Z$ are $II$, $II^*$ and $S$ is one of $S(E_8)$, $S(E_7)$, $S(E_6)$, $S(D_5)$, $S(A_4)$, or $S(A_1+A_2)$.
        \item[\textup{(3-2)}] $\Char(k)=5$, the singular fibers of $Z$ are $I_5$, $I_5$, $II$, and $S=S(2A_4)$.
    \end{enumerate}
\end{lem}

Note that cases (1) and (2) are erroneously omitted in \cite[Lemma B.9]{Lac}.
The same argument applies even if $C$ has a simple node.

\begin{lem}\label{lem:tiger-node}
    Let $\Char(k) \neq 2,3$, $S$ be a Du Val del Pezzo surface of rank one, and $C \in |-K_S|$ such that $C$ has a simple node and $C \subset S^\circ$.
    Then one of the following holds.
    \begin{enumerate}
        \item[\textup{(1)}] $S \cong \PP^2_k$ and $C$ is a nodal cubic.
        \item[\textup{(2)}] $S \cong S(A_1)$ and $C$ is a nodal quartic in $S(A_1) \subset \PP^3_k$.
        \item[\textup{(3)}] $S$ is obtained by blowing down $(-1)$-curves on an extremal rational elliptic surface $Z$ and $C_Z$ is a singular fiber of type $I_1$.
    \end{enumerate}
    In Case (3), one of the following holds.
    \begin{enumerate}
        \item[\textup{(3-1)}] $K_S^2=1$ and $S$ is one of the following: $S(E_8)$, $S(E_7+A_1)$, $S(E_6+A_2)$, $S(D_8)$, $S(2A_4)$, $S(D_5+A_3)$, $S(A_8)$, $S(A_7+A_1)$, or $S(A_5+A_2+A_1)$. 
        \item[\textup{(3-2)}] $K_S^2=2$ and $S$ is one of the following: $S(E_7)$, $S(A_1+D_6)$, $S(A_1+2A_3)$, or $S(A_2+A_5)$.
        \item[\textup{(3-3)}] $K_S^2=3$ and $S$ is one of the following: $S(E_6)$, $S(A_1+A_5)$, or $S(3A_2)$.
        \item[\textup{(3-4)}] $K_S^2=4$ and $S$ is one of the following: $S(D_5)$ or $S(2A_1+A_3)$.
        \item[\textup{(3-5)}] $K_S^2 \geq 5$ and $S$ is one of the following: $S(A_4)$ or $S(A_1+A_2)$.
    \end{enumerate}
    Moreover, we have $S \neq S(2A_4)$ when $\Char(k) =5$.
\end{lem}

\begin{proof}
    The assertions follow by the same method as in the proof of \cite[Lemma B.9]{Lac}.
    Note that the last assertion follows from the absence of extremal rational elliptic surfaces with singular fibers $I_5$, $I_5$, $I_1$, $I_1$ in $\Char(k)=5$ \cite[Theorem 4.1]{Lac}.
\end{proof}

The assumption $\Char(k) \neq 2,3$ is essential for Lemmas \ref{lem:tiger-cusp} and \ref{lem:tiger-node}.
This is because, when $\Char(k) = 2,3$, there are rational quasi-elliptic surfaces over $k$, and the classification of the singular fibers of extremal rational elliptic surfaces over $k$ is significantly different from that over $\C$ (see \cite{Lang1, Lang2, Ito1, Ito2} for details).

\section*{Acknowledgements}
The author would like to express his sincere gratitude to Takashi Kishimoto and the other editors of the proceedings of the conference ``Varieties with Boundaries'' for giving him the opportunity to contribute this paper, despite not having participated in the conference.
The author is supported by JSPS KAKENHI Grant Number JP21K13768.

\bibliography{hoge.bib}
\bibliographystyle{alpha}

\end{document}